\documentclass[11pt,reqno]{amsart}

\usepackage[T1]{fontenc}
\usepackage[utf8]{inputenc}
\usepackage{scalefnt,letltxmacro}
\usepackage{amsfonts,dsfont,mathrsfs}

\usepackage{amssymb,amsthm,amscd,empheq,amsmath,cancel}
\usepackage{mathtools,mathbbol,bm,bbm,nicefrac,scalerel}
\usepackage[thinc]{esdiff}
\usepackage{physics}
\usepackage{siunitx}
\usepackage[parfill]{parskip}
\usepackage{afterpage}
\ifdefined\ispres\else
\usepackage{enumitem}
\newlist{nenumerate}{enumerate}{3}
\setlist[nenumerate]{label=\arabic*., leftmargin=*}
\newlist{uenumerate}{enumerate}{3}
\setlist[uenumerate]{label=\arabic*., wide, labelwidth=!, labelindent=0pt}
\fi
\usepackage{framed}
\usepackage{xspace}

\ifdefined\ispres\else
\usepackage[dvipsnames]{xcolor}
\definecolor{shadecolor}{gray}{0.9}

\fi

\usepackage{graphicx}
\usepackage{pgf,wrapfig,epstopdf,psfrag,epsfig}
\usepackage[labelfont=bf]{caption}
\usepackage{floatrow}
\usepackage[format=hang]{subcaption}

\usepackage{booktabs, array, makecell,supertabular}

\usepackage{listings}
\usepackage{fancyvrb}
\fvset{fontsize=\small}

\usepackage[framemethod=TikZ]{mdframed}
\definecolor{theoremcolor}{rgb}{0.94, 0.97, 1.0}  %
\mdfsetup{
	backgroundcolor=shadecolor,
	linewidth=0.5pt,
	linewidth=0pt,
}

\ifdefined\nonewproofenvironments\else
\ifdefined\ispres\else
\ifdefined\mdproofenvrionments
\newmdtheoremenv{theorem}{Theorem}
\mdtheorem{lemma}{Lemma}
\mdtheorem{proposition}{Proposition}
\mdtheorem{corollary}{Corollary}
\else
\theoremstyle{plain}
\newtheorem{theorem}{Theorem}
\newtheorem{lemma}{Lemma}[theorem]
\newtheorem{proposition}[theorem]{Proposition}

\fi

\theoremstyle{definition}
\ifdefined\mdproofenvrionments
\mdtheorem{definition}{Definition}
\mdtheorem{fact}{Fact}
\mdtheorem{conjecture}{Conjecture}
\mdtheorem{example}{Example}
\else
\newtheorem{definition}{Definition}

\newtheorem{example}{Example}
\fi

\renewenvironment{proof}{\noindent\textbf{Proof.}\hspace*{.3em}}{\qed\\}

\newenvironment{proof-sketch}{\noindent\textbf{Proof Sketch}
\hspace*{1em}}{\qed\bigskip\\}
\newenvironment{proof-idea}{\noindent\textbf{Proof Idea}
\hspace*{1em}}{\qed\bigskip\\}
\newenvironment{proof-attempt}{\noindent\textbf{Proof Attempt}
\hspace*{1em}}{\qed\bigskip\\}

\theoremstyle{remark}

\fi

\theoremstyle{remark}

\fi
\numberwithin{equation}{section}

\usepackage[toc,page]{appendix}
\ifdefined\ispres\else
\usepackage[colorlinks,linktoc=all]{hyperref}
\hypersetup{citecolor=MidnightBlue}
\hypersetup{linkcolor=MidnightBlue}
\hypersetup{urlcolor=MidnightBlue}
\fi
\usepackage[nameinlink]{cleveref}
\creflabelformat{equation}{#2\textup{#1}#3}  %

\ifdefined\ispres\else\ifdefined\hypersetup
		\hypersetup{ %
			pdftitle={},
			pdfauthor={},
			pdfsubject={},
			pdfkeywords={},
			pdfborder=0 0 0,
			pdfpagemode=UseNone,
			colorlinks=true,
			linkcolor=MidnightBlue,
			citecolor=MidnightBlue,
			filecolor=MidnightBlue,
			urlcolor=MidnightBlue,
			pdfview=FitH}

		\ifdefined\isaccepted \else
			\hypersetup{pdfauthor={Anonymous Submission}}
		\fi
	\fi\fi

\usepackage[open=true]{bookmark}

\usepackage
[acronym,smallcaps,nowarn,section,nogroupskip,nonumberlist]{glossaries}
\glsdisablehyper{}

\makeatletter
\newcommand\BeraMonottfamily{%
	\def\fvm@Scale{0.85}%
	\fontfamily{fvm}\selectfont%
}
\makeatother

\lstdefinestyle{mystyle}{
	commentstyle=\color{OliveGreen},
	numberstyle=\tiny\color{black!60},
	stringstyle=\color{BrickRed},
	basicstyle=\BeraMonottfamily\normalsize,
	breakatwhitespace=false,
	breaklines=true,
	captionpos=b,
	keepspaces=true,
	numbers=none,
	numbersep=5pt,
	showspaces=false,
	showstringspaces=false,
	showtabs=false,
	tabsize=2,
	frame=lines
}
\usepackage[frozencache,cachedir=.]{minted}
\newminted[python]{python}{
  escapeinside=||,
  mathescape=true
 }
\usepackage[makestderr]{pythontex}
\restartpythontexsession{\thesection}

\usepackage{tikz}
\tikzset{
	invisible/.style={opacity=0},
	visible on/.style={alt=#1{}{invisible}},
	alt/.code args={<#1>#2#3}{%
			\alt<#1>{\pgfkeysalso{#2}}{\pgfkeysalso{#3}}
		},
}
\usepackage{tikz-dependency}
\usetikzlibrary{matrix,positioning}
\usetikzlibrary{arrows,patterns,shapes.arrows}
\pgfdeclarelayer{edgelayer}
\pgfdeclarelayer{nodelayer}
\pgfsetlayers{edgelayer,nodelayer,main}

\definecolor{hexcolor0xbfbfbf}{rgb}{0.749,0.749,0.749}

\tikzset{>=latex}
\tikzstyle{none}   = [inner sep=0pt]
\tikzstyle{line}   = [ -, thick, shorten <=1pt, shorten >=1pt ]
\tikzstyle{arrow}  = [ ->, thick, shorten <=1pt, shorten >=1pt ]
\tikzstyle{ardash} = [ dashed, ->, thick, shorten <=1pt, shorten >=1pt ]

\tikzstyle{empty}=[circle,opacity=0.0,text opacity=1.0,inner sep=0pt]
\tikzstyle{box}=[rectangle,fill=White,draw=Black]
\tikzstyle{filled}=[circle,thick,fill=hexcolor0xbfbfbf,draw=Black]
\tikzstyle{hollow}=[circle,thick,fill=White,draw=Black]
\tikzstyle{param}=[rectangle,fill=Black,draw=Black,inner sep=0pt,minimum
width=4pt,minimum height=4pt]
\tikzstyle{paramhollow}=[rectangle,thick,fill=White,draw=Black,inner sep=0pt,minimum
width=4pt,minimum height=4pt]

\usepackage{pgfplots}				    %
\pgfplotsset{compat=newest}
\pgfplotsset{plot coordinates/math parser=false}
\DeclareUnicodeCharacter{2212}{−}
\usepgfplotslibrary{groupplots,dateplot}
\newlength\figureheight
\newlength\figurewidth
\newlength\figureheightsmall
\newlength\figurewidthsmall
\definecolor{POSTcolor}{rgb}{0.48, 0.20, 0.58} %
\definecolor{Qcolor}{rgb}{0.00, 0.53, 0.22} %
\DeclareRobustCommand{\frac}[3][0pt]{%
  {\begingroup\hspace{#1}#2\hspace{#1}\endgroup\over\hspace{#1}#3\hspace{#1}}}

\newacronym{WLOG}{wlog}{without loss of generality}
\newacronym{LHS}{lhs}{left hand side}
\newacronym{RHS}{rhs}{right hand side}
\newacronym{PMF}{pmf}{Probability mass function}
\newacronym{PDF}{pdf}{Probability density function}
\newacronym{VI}{vi}{variational inference}
\newacronym{KL}{kl}{Kullback-Leibler}
\newacronym{RL}{rl}{reinforcement learning}
\newacronym{ELBO}{elbo}{\emph{evidence lower bound}}
\newacronym{MCMC}{mcmc}{Markov chain Monte Carlo}
\newacronym{OT}{ot}{optimal transport}

\renewcommand{\d}[1]{\ensuremath{\operatorname{d}\!{#1}}}

\newcommand{\DP}[2]{{\left\langle #1, #2\right\rangle}}

\def\tinycitep*#1{{\tiny\citep*{#1}}}
\def\tinycitealt*#1{{\tiny\citealt*{#1}}}
\def\tinycite*#1{{\tiny\cite*{#1}}}
\def\smallcitep*#1{{\scriptsize\citep*{#1}}}
\def\smallcitealt*#1{{\scriptsize\citealt*{#1}}}
\def\smallcite*#1{{\scriptsize\cite*{#1}}}

\newcommand{\defeq}{\vcentcolon=}

\DeclareMathOperator*{\argmin}{\mathsf{argmin}}

\def\R{\mathbbm{R}}

\def\<{\left\langle} %
\def\>{\right\rangle}

\def\HH{\mathsf{H}}

\def\HHr{\HH^{\textsc{r}}}

\usepackage{twoopt}
\newcommandtwoopt{\LD}[4][\lambda][\varphi]{\ensuremath{\mathbf{L}_{#1, #2}\left[
#3 : #4 \right] }}

\newacronym{MD}{MD}{mirror descent}
\newacronym{LD}{LD}{logarithmic divergence}
\newacronym{MAP}{MAP}{maximum a-posteriori}
\newacronym{MLE}{MLE}{maximum likelihood estimation}
\newacronym{GMT}{GMT}{gumbel-max trick}
\newacronym{GST}{GST}{gumbel-softmax trick}
\newacronym{STGS}{stgs}{straight-through gumbel-softmax}

\newacronym{FOC}{foc}{first-order optimality condition}
\newacronym{PL}{pl}{Polyak-L\"{o}jasiewicz}
\newacronym{CRT}{crt}{Cram\'{e}r-Rao theorem}
\newacronym{ODE}{ode}{ordinary differential equation}
\newacronym{MF}{mf}{mirror flow}
\newacronym{MDA}{mda}{mirror descent algorithm}
\newacronym{MDD}{mdd}{mirror descent dynamic}
\newacronym{CMF}{cmf}{conformal mirror flow}
\newacronym{CMD}{cmd}{conformal mirror descent}
\newacronym{CBD}{cbd}{conformal Bregman descent}
\newacronym{NGD}{ngd}{natural gradient descent}
\newacronym{NGF}{ngf}{natural gradient flow}
\newacronym{IID}{iid}{indepedent and identically distributed}

\newacronym{JSD}{JSD}{Jensen-Shannon divergence}
\newacronym{RD}{RD}{R\'{e}nyi divergence}
\newacronym{MSE}{MSE}{mean squared error}
\newacronym{FY}{FY}{Fenchel-Young}

\def\p{\theta}
\def\d{\eta}

\def\x{x}

\usepackage{resizegather}
\usepackage{ifthen}
\newboolean{long}

\begin{document}

\title[Conformal Mirror Descent]{Conformal mirror descent with logarithmic divergences}

\author[Kainth]{Amanjit Singh Kainth}
\address{Department of Computer Science, University of Toronto}
\email{amanjitsk@cs.toronto.edu}

\author[Wong]{Ting-Kam Leonard Wong}
\address{Department of Statistical Sciences, University of Toronto}
\email{tkl.wong@utoronto.ca}

\author[Rudzicz]{Frank Rudzicz}
\address{Department of Computer Science, University of Toronto}
\email{frank@spoclab.com}

\begin{abstract}
The logarithmic divergence is an extension of the Bregman divergence
motivated by optimal transport and a generalized convex duality, and satisfies
many remarkable properties. Using the geometry induced by the logarithmic
divergence, we introduce a generalization of continuous time mirror descent that
we term the conformal mirror descent. We derive its dynamics under a generalized
mirror map, and show that it is a time change of a corresponding Hessian
gradient flow. We also prove convergence results in continuous time. We apply
the conformal mirror descent to online estimation of a generalized exponential
family, and construct a family of gradient flows on the unit simplex via the
Dirichlet optimal transport problem.
\end{abstract}


\maketitle

\section{Introduction}%
\label{sec:introduction}

Information geometry provides not only powerful tools for studying spaces of
probability distributions, but also a wide range of geometric structures that are
useful for various challenges in data science \cite{an00, a16, ajvs17}. The Bregman
divergence \cite{b67} plays a key role in the theory and application of
information geometry. It is the canonical divergence of the dually flat geometry
\cite{na82} which arises naturally in exponential families \cite{bmdg05}, and
can serve as a loss function in statistical estimation and optimal control
\cite{EK22}. The Bregman divergence is especially tractable in applied settings,
as it is closely connected to convex duality and satisfies a generalized
Pythagorean theorem which greatly simplifies the analysis of Bregman
projections. Among the many applications of Bregman divergences, we mention
clustering \cite{bmdg05}, exponential family principal component analysis
\cite{cds02} as well as boosting and logistic regression \cite{css02,
murata2004information}.

We present a generalization of {\it mirror descent}
\cite{NY83,am03} which is a popular first-order iterative optimization
algorithm. Mirror descent is a gradient descent algorithm where a Bregman
divergence serves as a proximal function. A suitable convex generating function
may be chosen to exploit the geometry of the problem. The update step
\eqref{eqn:MD.update} of mirror descent involves a change of coordinates using
the so-called {\it mirror map} which corresponds to the information-geometric
dual parameter. In the continuous time limit, mirror descent can be represented
as a Riemannian gradient flow with respect to the Hessian metric induced by the
given Bregman divergence \cite{ABB04, rm15}. The basic ideas are reviewed in
Sections \ref{sec:convex.duality} and
\ref{sub:mirror_descent_and_riemannian_gradient_flow}.

Our generalization, termed the {\it conformal mirror descent}, is based on the
theory of {\it logarithmic divergences}  \cite{pw16, pw18, w18,w19,WZ21}. In
many senses, the logarithmic divergence may be regarded as a canonical
deformation of the Bregman divergence. Just as the Bregman divergence captures
the dually flat geometry, the logarithmic divergence is a canonical divergence
for a dually projectively flat statistical manifold with {\it constant} nonzero
sectional curvature, and also satisfies a generalized Pythagorean theorem
\cite{w18}. Moreover, the logarithmic divergence under divisive normalization
leads to a deformed exponential family, which is closely related to the
$q$-exponential family in statistical physics \cite{naudts2011generalised},
while recovering natural analogues of intrinsic information-geometric properties
of the exponential family in the deformed case \cite{w18, WZ21}. For example,
the Kullback-Leibler (KL) divergence (which is the Bregman divergence of the
cumulant generating function) becomes the R\'{e}nyi divergence, and the dual
variable can be interpreted as an escort expectation. Another appealing property
is that the logarithmic divergence is associated with a generalized convex
duality motivated by optimal transport \cite{v03, v08}. Following \cite{WZ21},
we call it the {\it $\lambda$-duality}, where $\lambda \neq 0$ is the curvature
parameter. It was recently shown \cite{WY19B} that  the dualistic geometry in
information geometry can be naturally embedded in the pseudo-Riemannian geometry
of optimal transport \cite{km10} using the framework of $c$-divergence, under
which divergences are induced by optimal transport maps. Bregman and logarithmic
divergences are special cases corresponding to particular cost functions
\cite{pw18, w18}. In Section \ref{sub:logarithmic_divergences}, we review
properties of $\lambda$-duality and logarithmic divergences that are needed in
this paper. Further results about $\lambda$-duality and its relation with convex
duality can be found in \cite{zhang2022lambda}.

In Section \ref{sec:logarithmic_mirror_descent}, we formulate the conformal
mirror descent in continuous time as a Riemannian gradient flow, where the
underlying metric is induced by a logarithmic divergence. We call it the {\it
conformal mirror descent} because the metric can be shown to be a conformal
transformation of a Hessian metric. This implies that the conformal mirror
descent is, in continuous time, a time change of mirror descent. We also derive
explicit dynamics of the gradient flow under the $\lambda$-mirror map
corresponding to the logarithmic divergence and related convergence results. The
$\lambda$-duality suggests many new generating functions that are potentially
useful in various applications.

We give two applications to demonstrate the utility of our conformal mirror
descent. In Section
\ref{sub:online_parameter_estimation_for_generalized_exponential_family}, we
consider online estimation of the $\lambda$-exponential family introduced in
\cite{WZ21}, and derive an elegant online natural gradient update which
generalizes the one for the exponential family \cite{rm15}. {\it Dirichlet
optimal transport} on the unit simplex \cite{pw16, pw18, pw18b} is one of the
original motivations of the theory of logarithmic divergences (and corresponds to
the case $\lambda = -1$). Expressing the $(-1)$-mirror map in terms of the
Dirichlet optimal transport map, we derive in Section \ref{sub:Dirichlet}, an
interesting family of gradient flows on the unit simplex.

Finally, in Section \ref{sec:conclusion} we discuss our contributions in the
context of related literature, and propose several directions for future
research.

{\it Notation}: We use superscripts to denote components of vectors, e.g.,
$\theta = (\theta^1, \ldots, \theta^d)$. In computations, we regard $\theta$ as
a column vector and write $\theta = \begin{bmatrix} \theta^1 & \cdots
& \theta^d\end{bmatrix}^{\top}$, where $\top$ denotes transposition. The
Euclidean gradient $\nabla f(\theta) = \nabla_{\theta} f(\theta)$ of
a real-valued function $f$ is also regarded as a column vector. Due to the
difficulty of unifying notations in different settings, in this paper we do {\it
not} adopt the Einstein summation convention.

\section{From convex duality to $\lambda$-duality} \label{sec:background}
\subsection{Convex duality and Bregman divergence} \label{sec:convex.duality}
We begin by reviewing convex duality and Bregman divergence, which are at the
core of classical information geometry \cite{an00, a16} (also see
\cite{amari2021information} for a recent overview). Let $\phi$ be a lower
semicontinuous convex function on $\mathbbm{R}^d$. Its convex conjugate is
defined by $\phi^*(y) = \sup_{x \in \mathbbm{R}^d} \left\{\DP{x}{y}
- \phi(x) \right\}$, where $\DP{\cdot}{\cdot}$ denotes the Euclidean inner
product. Then $\phi^*$ is also lower semi-continuous and convex, and we have
$\phi^{**} = (\phi^*)^* = \phi$. From the definition of $\phi^*$, for any $x, y \in
\R^d$ we have
\begin{equation} \label{eqn:Fenchel.inequality}
\phi(x) + \phi^*(y) - \DP{x}{y} \geq 0,
\end{equation}
and equality holds if and only if $y$ is a subgradient of $\phi$ at $x$.

Let $\Theta \subset \mathbbm{R}^d$ be an open convex set and let $\phi: \Theta
\rightarrow \mathbbm{R}$ be a smooth convex function whose Hessian $\nabla^2
\phi(\theta)$ is everywhere positive definite. We call such a $\phi$ a {\it
Bregman generator}. The {\it Bregman divergence} of $\phi$, regarded as
a generalized distance, is defined for $\theta, \theta' \in \Theta$ by
\begin{equation} \label{eqn:phi.Bregman}
{\bf B}_{\phi}[\theta : \theta'] = \left(\phi(\theta) - \phi(\theta')\right)
- \DP{\nabla \phi(\theta')}{\theta - \theta'}.
\end{equation}
Under the stated conditions, $\nabla \phi$ is a diffeomorphism from $\Theta$
onto its range. We call $\theta$ the {\it primal variable} and $\zeta = \nabla
\phi(\theta)$ the {\it dual variable}.\footnote{We reserve the symbol $\eta$ for
the dual variable under the $\lambda$-duality; see
\eqref{eqn:L.alpha.mirror.map}.} The inverse transformation is given by $\theta
    = \nabla \phi^*(\zeta)$. Then, we may express
\eqref{eqn:phi.Bregman} in {\it self-dual form} by
\begin{equation} \label{eqn:Bregman.self.dual}
\begin{split}
{\bf B}_{\phi}[\theta: \theta'] = {\bf B}_{\phi^*}[\zeta' : \zeta] = \phi(\theta)
+ \phi^*(\zeta') - \DP{\theta}{\zeta'},
\end{split}  %
\end{equation}
which is closely related to the Fenchel-Young inequality
\eqref{eqn:Fenchel.inequality}.

Conjugation, which characterizes convex duality, is defined in terms of the
pairing function $c(x, y) = -\DP{x}{y}$. It turns out that much of the above can
be generalized. For a general $c$, called the cost function in the context of
optimal transport \cite{v03, v08}, we can define the $c$-conjugate of a function
$\varphi$ by $\varphi^{(c)}(y) = \sup_{x} \{ -c(x, y) - \varphi(x) \}$.
A function $\varphi$ is said to be $c$-convex if it is the $c$-conjugate of some
function $\varphi$, i.e., $\varphi = \varphi^{(c)}$. We have the following
analogue of the Fenchel-Young inequality  \eqref{eqn:Fenchel.inequality}:
\begin{equation} \label{eqn:c.FY.inequality}
\varphi(x) + \varphi^{(c)}(y) + c(x, y) \geq 0.
\end{equation}
If equality holds, we call $y$ a $c$-subgradient of $\phi$ at $x$. If this $y$
is unique, we call it the $c$-gradient. Under suitable conditions,
a Monge-Kantorovich optimal transport problem can be solved by an optimal
transport map, which can be expressed as the $c$-gradient of some $c$-convex
potential $\varphi$. The inequality~\eqref{eqn:c.FY.inequality} can be used to
define a {\it $c$-divergence} on the graph of the optimal transport map
\cite{WY19B}. The $\lambda$-duality \cite{WZ21} is the generalized convex
duality based on the {\it logarithmic cost}
\begin{equation} \label{eqn:log.cost}
c_{\lambda}(x, y) = \frac{-1}{\lambda} \log(1 + \lambda \DP{x}{y}), \quad \lambda \neq 0,
\end{equation}
which is mathematically tractable and has remarkable properties.\footnote{In our
applications $x$ and $y$ only vary in respective domains such that $1 + \lambda
\DP{x}{y} > 0$, so the logarithm in \eqref{eqn:log.cost} is well defined.} We
recover the usual convex duality in the limiting case $ \lim_{\lambda \to 0}
c_{\lambda}(x, y) = - \DP{x}{y}$.

\subsection{Mirror descent} \label{sub:mirror_descent_and_riemannian_gradient_flow}
Consider the minimization problem $\min_{\theta \in \Theta} f(\theta)$ where $f: \Theta
\rightarrow \mathbbm{R}$ is assumed to be differentiable.
Let $\phi : \Theta \rightarrow \mathbbm{R}$ be a Bregman generator as in Section
\ref{sec:convex.duality}. It induces the {\it mirror map} $\zeta
= \nabla_{\theta} \phi(\theta)$. For clarity, we use $\nabla_{\theta}$ to
indicate that the gradient is taken with respect to $\theta$. The mirror descent
algorithm minimizes $f$ by iterating the update
\begin{equation} \label{eqn:MD.update}
\zeta_{k+1} = \zeta_k - \delta \nabla_{\theta} f(\theta_k),
\end{equation}
where $\delta = \delta_k > 0$ is the learning rate which may depend on $k$. We
obtain $\theta_{k+1}$ by applying the inverse mirror map, i.e., $\theta_{k+1}
= \nabla_{\zeta} \phi^*(\zeta_{k+1})$. Thus, we require that both $\nabla \phi$
and $\nabla \phi^*$ can be computed for implementing the algorithm. Letting
$\phi(\theta) = \frac{1}{2}|\theta|^2 = \frac{1}{2} \DP{\theta}{\theta}$
recovers Euclidean gradient descent since in this case, $\zeta = \nabla_{\theta}
\frac{1}{2}|\theta|^2 = \theta$. In general,~\eqref{eqn:MD.update} requires an
extra projection step when the right hand side is outside $\Theta$. The
(unconstrained) update \eqref{eqn:MD.update} is equivalent to the update of
a {\it Bregman proximal method}, namely
\begin{equation} \label{eqn:MD.update.equiv}
\theta_{k+1} = \argmin_{\theta \in \Theta} \left\{ f(\theta_k) +  \DP{\nabla_{\theta}
 f(\theta_k)}{\theta - \theta_k} + \frac{1}{\delta} {\bf B}_{\phi}[\theta : \theta_k]
 \right\}.
\end{equation}
It is easy to verify that the first order condition of
\eqref{eqn:MD.update.equiv} can be expressed as \eqref{eqn:MD.update}.
Geometrically, $\theta_{k+1}$ minimizes a linear approximation of $f$  over
a Bregman ball based at $\theta_k$.

Further insights can be obtained by studying the continuous time limit as done
in \cite{rm15} and \cite{GWS21}. The Bregman divergence admits the quadratic
approximation
\begin{equation} \label{eqn:Bregman.Riemannian}
{\bf B}_{\phi}[ \theta + \Delta \theta : \theta] = \frac{1}{2} (\Delta
\theta)^{\top} G_0(\theta) (\Delta \theta) + O(|\Delta \theta|^3),
\end{equation}
where $G_0(\theta) = \nabla_{\theta}^2 \phi(\theta)$ is a {\it Hessian}
Riemannian metric (when expressed under the primal $\theta$-coordinates) and
induces the {\it Riemannian gradient} $\mathrm{grad}_{G_0} f = G_0^{-1}
\nabla_{\theta} f$. See \cite{shima2007geometry} for an in-depth geometric study
of Hessian manifolds. Letting $\delta \rightarrow 0$ in \eqref{eqn:MD.update} or
\eqref{eqn:MD.update.equiv} and scaling time appropriately, one obtains a {\it
Hessian Riemannian gradient flow} \cite{ABB04}:
\begin{equation} \label{eqn:MD.flow}
\begin{split}
\dv{t} \theta_t = -\mathrm{grad}_{G_0} f(\theta_t), \quad \text{or equivalently}
\quad  \dv{t} \zeta_t = -\nabla_{\theta} f(\theta_t).
\end{split}
\end{equation}
Naturally, one may consider other metrics to obtain generalizations of mirror
descent (see \cite{GWS21} for a discussion). In this paper, we use the
Riemannian metric induced by the logarithmic divergence, which is particularly
tractable.

\subsection{$\lambda$-duality and logarithmic divergence}
\label{sub:logarithmic_divergences}
Following the treatment of \cite{WZ21}, we introduce the $\lambda$-duality which
utilizes the logarithmic cost function $c_{\lambda}$ defined by
\eqref{eqn:log.cost}. In general, $c$-convex functions and $c$-gradients are
difficult to characterize explicitly. Remarkably, for the logarithmic cost
function $c_{\lambda}$, it is possible to relate $c_{\lambda}$-convex with usual
convexity and express the $c_{\lambda}$-gradient in terms of the usual gradient.
The following definition summarizes the generalized convexity notion and the
required regularity conditions needed for our applications. Throughout, we let
$\lambda \neq 0$ be a fixed constant.

\begin{definition} [Regular $c_{\lambda}$-convex function and
  $c_{\lambda}$-gradient] \label{def:regular.c.lambda}
Let $\Theta \subset \R^d$ be an open convex set. A smooth function $\varphi:
\Theta \rightarrow \mathbbm{R}$ is said to be regular $c_{\lambda}$-convex if
$\Phi_{\lambda} = \frac{1}{\lambda} (e^{\lambda \varphi} - 1)$ is convex,
$\nabla_{\theta}^2 \Phi_{\lambda}$ is positive definite and $1 - \lambda
\DP{\nabla_{\theta} \varphi(\theta)}{\theta} > 0$. Given such a function
$\varphi$, we define its $c_{\lambda}$-gradient by
\begin{equation} \label{eqn:L.alpha.mirror.map}
\nabla^{(\lambda)}_{\theta} \varphi(\theta) = \frac{1}{1 - \lambda
\DP{\nabla_{\theta} \varphi(\theta)}{\theta}} \nabla_{\theta}
\varphi(\theta).
\end{equation}
\end{definition}

We also call $\nabla^{(\lambda)}_{\theta} \varphi$ the {\it $\lambda$-mirror
map}. Under the stated conditions, it can be shown that
$\nabla^{(\lambda)}_{\theta} \varphi$ is a diffeomorphism from $\Theta$ onto its
range $H$;\footnote{$H$ is the uppercase of the Greek letter $\eta$.} we call
$\eta = \nabla_{\theta}^{(\lambda)} \varphi(\theta)$ the dual variable in this
context.

\begin{table}[t!]
\begin{tabular}{c|c|c|c}
$\lambda$ range & $\Theta$ & $\varphi(\theta)$ & $\eta = \nabla^{(\lambda)}
\varphi(\theta)$ \\
\hline
$(-2, \infty)$ & $(0, \infty)$ & $-\frac{1}{2} \log \theta$ & $\frac{-1}{2
+ \lambda} \frac{1}{\theta}$ \\
$(0, \infty)$ & $(-\infty, \frac{1}{\lambda})$ & $\theta$ & $\frac{1}{1
- \lambda \theta}$ \\
$\mathbbm{R} \setminus \{0\}$ & $(\frac{-1}{\sqrt{|\lambda|}},
\frac{1}{\sqrt{|\lambda|}})$ & $\frac{1}{2}\theta^2$ & $\frac{\theta}{1
- \lambda \theta^2}$ %
\end{tabular}
\caption{Examples of regular $c_{\lambda}$-convex functions on the real line and
their corresponding $\lambda$-mirror maps.}\label{tab:examples}
\end{table}

In a nutshell, instead of convex functions, we use functions $\varphi$ such that
$\Phi_{\lambda} = \frac{1}{\lambda} (e^{\lambda \varphi} - 1)$ are convex, and
replace the usual gradient by the $\lambda$-mirror map. Some examples of regular
$c_{\lambda}$-convex functions are given in Table \ref{tab:examples}.

Henceforth we let $\varphi$ be a regular $c_{\lambda}$-convex function. Let
$\psi$ be the {\it $c_{\lambda}$-conjugate} defined by
\[
\psi(\eta) = \sup_{\theta \in \Theta} \left\{ -c_{\lambda}(\theta, \eta)
- \varphi(\theta) \right\}.
\]
Then, for $\theta \in \Theta$ we have
\[
\varphi(\theta) = \sup_{\eta \in H} \left\{ -c_{\lambda}(\theta, \eta)
- \psi(\eta) \right\}.
\]
Moreover, $1 + \lambda \DP{\theta}{\eta'} > 0$ for any $\theta \in \Theta$ and
$\eta' \in H$, and for $\eta = \nabla^{(\lambda)} \varphi(\theta)$ we have the
identity
\begin{equation} \label{eqn:c.Fenchel.Young.equality}
\varphi(\theta) + \psi(\eta) + c_{\lambda}(\theta, \eta) = 0.
\end{equation}
Thus $\varphi$ and $\psi$ satisfy a generalized Legendre-like duality with
respect to the cost function $c_{\lambda}$. The inverse $\lambda$-mirror map is
given by $\theta = \nabla_{\eta}^{(\lambda)} \psi(\eta)$.

We use $\varphi$ to define a {\it $\lambda$-logarithmic divergence} which is
different from the Bregman divergence. For completeness, we review the main
idea. Recall that $\Phi_{\lambda} = \frac{1}{\lambda} (e^{\lambda \varphi} - 1)$
is (strictly) convex. By convexity, for $\theta, \theta' \in \Theta$ we have
\[
\Phi(\theta') +   \DP{\nabla \Phi(\theta')}{\theta - \theta'} \leq \Phi(\theta).
\]
Expressing this inequality in terms of $\varphi$, we have, using the chain rule,
\begin{equation} \label{eqn:L.lambda.derive.1}
\begin{split}
&\quad \frac{1}{\lambda} e^{\lambda \varphi(\theta')} + e^{\lambda
\varphi(\theta')} \nabla \varphi(\theta') \cdot (\theta - \theta') \leq
\frac{1}{\lambda} e^{\lambda \varphi(\theta)}\\
&\Rightarrow \frac{1}{\lambda}(1 + \lambda \nabla \varphi(\theta') \cdot (\theta
- \theta')) \leq \frac{1}{\lambda} e^{\lambda (\varphi(\theta)
- \varphi(\theta')}.
\end{split}
\end{equation}
Now there are two cases depending on the sign of $\lambda$, but the resulting
expression is the same. Here, we consider the case $\lambda < 0$ and the other
case is similar. From \eqref{eqn:L.lambda.derive.1}, we have
\begin{equation} \label{eqn:L.lambda.derive.2}
\begin{split}
&\quad 1 + \lambda \nabla \varphi(\theta') \cdot (\theta - \theta')
\geq e^{\lambda(\varphi(\theta) - \varphi(\theta')} \\
&\Rightarrow \varphi(\theta') + \frac{1}{\lambda} \log(1 + \lambda \nabla
\varphi(\theta') \cdot (\theta - \theta')) \leq \varphi(\theta).
\end{split}
\end{equation}
Taking the difference yields the $\lambda$-logarithmic divergence. When
$\varphi$ is convex, letting $\lambda \rightarrow 0$
in~\eqref{eqn:L.lambda.derive.2} recovers the Bregman divergence (see Figure
\ref{fig:L.div.example}).

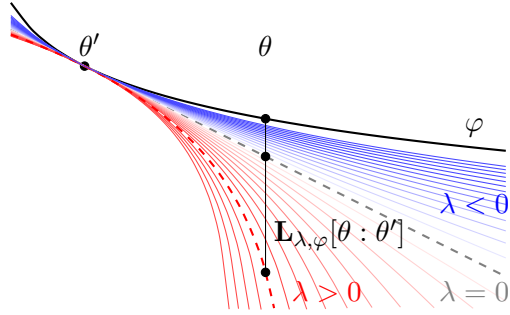
\begin{figure}[t!]
\centering
\begin{tikzpicture}[scale = 1.4]
\draw[domain = 0.3:5, thick, smooth, variable = \x] plot ({\x}, {-0.5*ln(\x)});
\draw[black, fill = black, radius = 0.04] (1, 0) circle;
\draw[domain = 0.3:5, gray, dashed, thick, smooth, variable = \x] plot ({\x},
  {-0.5*(\x - 1)});

\draw[domain = 0.3:5, blue!10, smooth, variable = \x] plot ({\x}, {-ln(1 + 0.1*0.5*(\x - 1))/0.1});
\draw[domain = 0.3:5, blue!18, smooth, variable = \x] plot ({\x}, {-ln(1 + 0.2*0.5*(\x - 1))/0.2});
\draw[domain = 0.3:5, blue!26, smooth, variable = \x] plot ({\x}, {-ln(1 + 0.3*0.5*(\x - 1))/0.3});
\draw[domain = 0.3:5, blue!32, smooth, variable = \x] plot ({\x}, {-ln(1 + 0.4*0.5*(\x - 1))/0.4});
\draw[domain = 0.3:5, blue!40, smooth, variable = \x] plot ({\x}, {-ln(1 + 0.5*0.5*(\x - 1))/0.5});
\draw[domain = 0.3:5, blue!48, smooth, variable = \x] plot ({\x}, {-ln(1 + 0.6*0.5*(\x - 1))/0.6});
\draw[domain = 0.3:5, blue!56, smooth, variable = \x] plot ({\x}, {-ln(1 + 0.7*0.5*(\x - 1))/0.7});
\draw[domain = 0.3:5, blue!64, smooth, variable = \x] plot ({\x}, {-ln(1 + 0.8*0.5*(\x - 1))/0.8});
\draw[domain = 0.3:5, blue!72, smooth, variable = \x] plot ({\x}, {-ln(1 + 0.9*0.5*(\x - 1))/0.9});
\draw[domain = 0.3:5, blue!72, smooth, variable = \x] plot ({\x}, {-ln(1 + 1.1*0.5*(\x - 1))/1.1});
\draw[domain = 0.3:5, blue!72, smooth, variable = \x] plot ({\x}, {-ln(1 + 1.2*0.5*(\x - 1))/1.2});
\draw[domain = 0.3:5, blue!72, smooth, variable = \x] plot ({\x}, {-ln(1 + 1.3*0.5*(\x - 1))/1.3});
\draw[domain = 0.3:5, blue!72, smooth, variable = \x] plot ({\x}, {-ln(1 + 1.4*0.5*(\x - 1))/1.4});

\draw[domain = 0.3:5, red!10, smooth, variable = \x] plot ({\x}, {ln(1 - 0.1*0.5*(\x - 1))/0.1});
\draw[domain = 0.3:4.69, red!18, smooth, variable = \x] plot ({\x}, {ln(1 - 0.2*0.5*(\x - 1))/0.2});
\draw[domain = 0.3:4.32, red!26, smooth, variable = \x] plot ({\x}, {ln(1 - 0.3*0.5*(\x - 1))/0.3});
\draw[domain = 0.3:4, red!32, smooth, variable = \x] plot ({\x}, {ln(1 - 0.4*0.5*(\x - 1))/0.4});
\draw[domain = 0.3:3.73, red!40, smooth, variable = \x] plot ({\x}, {ln(1 - 0.5*0.5*(\x - 1))/0.5});
\draw[domain = 0.3:3.495, red!48, smooth, variable = \x] plot ({\x}, {ln(1 - 0.6*0.5*(\x - 1))/0.6});
\draw[domain = 0.3:3.29, red!56, smooth, variable = \x] plot ({\x}, {ln(1 - 0.7*0.5*(\x - 1))/0.7});
\draw[domain = 0.3:3.11, red!64, smooth, variable = \x] plot ({\x}, {ln(1 - 0.8*0.5*(\x - 1))/0.8});
\draw[domain = 0.3:2.945, red!72, smooth, variable = \x] plot ({\x}, {ln(1 - 0.9*0.5*(\x - 1))/0.9});
\draw[domain = 0.3:2.677, red!72, smooth, variable = \x] plot ({\x}, {ln(1 - 1.1*0.5*(\x - 1))/1.1});
\draw[domain = 0.3:2.563, red!72, smooth, variable = \x] plot ({\x}, {ln(1 - 1.2*0.5*(\x - 1))/1.2});
\draw[domain = 0.3:2.462, red!72, smooth, variable = \x] plot ({\x}, {ln(1 - 1.3*0.5*(\x - 1))/1.3});
\draw[domain = 0.3:2.373, red!72, smooth, variable = \x] plot ({\x}, {ln(1 - 1.4*0.5*(\x - 1))/1.4});

\draw[domain = 0.3:2.8, red, dashed, thick, smooth, variable = \x] plot ({\x}, {ln(1 - 0.5*(\x - 1))});  %
\draw[domain = 0.3:5, blue!72, smooth, variable = \x] plot ({\x}, {-ln(1 + 0.5*(\x - 1))});  %

\draw[black, fill = black, radius = 0.04] (2.7183, -0.5) circle;  %
\draw[black, fill = black, radius = 0.04] (2.7183, -0.859) circle;  %
\draw[black, fill = black, radius = 0.04] (2.7183, -1.96) circle;  %
\draw[black] (2.7183, -1.96) -- (2.7183, -0.5);
\node[black, right] at (2.7, -1.6) {${\bf L}_{\lambda, \varphi}[\theta : \theta']$};

\node[gray, below] at (4.7, -1.95) {$\lambda = 0$};
\node[blue, above] at (4.7, -1.5) {$\lambda < 0$};
\node[red, below] at (3.3, -1.95) {$\lambda > 0$};

\node[black, above] at (4.7, -0.77) {$\varphi$};
\node[black, above] at (1.05, 0) {$\theta'$};
\node[black, above] at (2.7183, 0) {$\theta$};
\end{tikzpicture}
\caption{The $\lambda$-logarithmic divergence \eqref{eqn:log.div} is the error
term of a logarithmic first order approximation;
see~\eqref{eqn:L.lambda.derive.2}. We visualize it for $\varphi(\theta)
= \frac{-1}{2} \log \theta$ which is regular $c_{\lambda}$-convex on $(0,
\infty)$ for $\lambda > -2$. The red dashed curve shows the case $\lambda = 1$.}
\label{fig:L.div.example}
\end{figure}

\begin{definition}[$\lambda$-logarithmic divergence] \label{def:log.divergence}
Let $\varphi$ be a regular $c_{\lambda}$-convex function. We define its
$\lambda$-logarithmic divergence for $\theta, \theta' \in \Theta$ by
\begin{equation} \label{eqn:log.div}
{\bf L}_{\lambda, \varphi}[\theta : \theta'] = \varphi(\theta)
- \varphi(\theta') - \frac{1}{\lambda} \log(1 + \lambda \DP{ \nabla
\varphi(\theta')}{(\theta - \theta')}).
\end{equation}
\end{definition}

An important application of the logarithmic divergence is to some generalized
exponential families, where an appropriately defined potential function
$\varphi$ leads to the R\'{e}nyi divergence. See Section
\ref{sub:online_parameter_estimation_for_generalized_exponential_family}, where
we exploit this property in online parameter estimation.

Analogous to \eqref{eqn:Bregman.self.dual}, the $\lambda$-logarithmic divergence
admits the following self-dual representation which verifies that it is the
$c$-divergence of the cost $c_{\lambda}$:
\begin{equation} \label{eqn:log.self.dual}
{\bf L}_{\lambda, \varphi}[\theta : \theta'] = {\bf L}_{\lambda, \psi}[\eta'
: \eta] = \varphi(\theta) + \psi(\eta') - \frac{1}{\lambda} \log (1 + \lambda
\DP{\theta}{\eta'}).
\end{equation}
The logarithmic divergence can be justified by the remarkable properties
satisfied by the induced dualistic geometry $(g, \nabla, \nabla^*)$ \cite{pw18,
w18, WZ21}:
\begin{itemize}
    \item The primal and dual connections $(\nabla, \nabla^*)$ are {\it dually
      projectively flat}. In particular, the primal (resp.~dual) geodesics are
      time-reparameterized straight lines under the primal (resp.~dual)
      coordinate system.
    \item The sectional curvature of $\nabla$ and $\nabla^*$ with respect to $g$
      are everywhere {\it constant} and equal to $\lambda$.
    \item The {\it generalized Pythagorean theorem} holds for the
      $\lambda$-logarithmic divergence.
    \item Given a dualistic structure which is dually projectively projectively
      flat with constant (nonzero) sectional curvature, one can define (locally)
      a $\lambda$-logarithmic divergence which induces the given structure.
      Thus, the $\lambda$-logarithmic divergence can be regarded as a {\it
      canonical divergence}.
\end{itemize}

Letting $\lambda \rightarrow 0$ recovers well-known properties of the dually
flat geometry induced by a Bregman divergence.

To define the conformal mirror descent we will need the explicit form of the
metric $g$. We give two equivalent representations, both of which are useful in
our development. Under the primal coordinate system $\theta$, the coefficients
of $g$ are given by the matrix
\begin{equation} \label{eqn:L.alpha.metric}
G_{\lambda}(\theta) = \nabla^2_{\theta} \varphi(\theta) + \lambda (\nabla_{\theta}
\varphi(\theta))(\nabla_{\theta} \varphi(\theta))^{\top} = e^{-\lambda
\varphi(\theta)} \nabla_{\theta}^2 \Phi_{\lambda}(\theta).
\end{equation}
The first representation states that $G_{\lambda}$ is a rank-one correction of
the Hessian $\nabla^2_{\theta} \varphi$. The second representation states that
$G_{\lambda}$ is a {\it conformal transformation} of the Hessian metric
$\tilde{G}_0 = \nabla_{\theta}^2 \Phi_ {\lambda}$. That is, $G_{\lambda}$ is
a {\it conformal Hessian metric}. Geometrically, a conformal transformation
preserves the angles between tangent vectors but distorts their lengths. We note
that the last expression in \eqref{eqn:L.alpha.metric} may also be understood
via the identity
\begin{equation} \label{eqn:conformal.divergence}
{\bf L}_{\lambda, \varphi}[\theta : \theta'] = \frac{1}{-\lambda} \log \left(1
+ (-\lambda) e^{-\lambda \varphi(\theta)} {\bf B}_{\Phi_{\lambda}}[\theta:
\theta'] \right),
\end{equation}
which can be verified by a direct computation. It states that the
$\lambda$-logarithmic divergence is a monotone transformation of a {\it left
conformal Bregman divergence} \cite{NNA15}. See \cite{wy19} for more discussion
in this direction.

\section{Conformal mirror descent}
\label{sec:logarithmic_mirror_descent}
In this section, we present our first main contribution: a generalization of
continuous time mirror descent as the Riemannian gradient flow with respect to
a $\lambda$-logarithmic divergence. In Section \ref{sub:derivation_of_the_flow},
we define the flow and interpret it in two ways: (i) a mirror-like descent under
the $\lambda$-mirror map \eqref{eqn:L.alpha.mirror.map}, and (ii) a time change
of a Hessian gradient flow. It reduces to the continuous time mirror descent
\eqref{eqn:MD.flow} in the limit $\lambda \rightarrow 0$. Convergence results
are stated and proved in Section \ref{sec:flow.convergence.statements}.

\subsection{The flow and two representations}%
\label{sub:derivation_of_the_flow}
As described in Section \ref{sub:mirror_descent_and_riemannian_gradient_flow},
the usual (Bregman) mirror descent can be understood as (i) a Bregman proximal
method \eqref{eqn:MD.update.equiv}; or (ii) a (discretization of) the Hessian
gradient flow \eqref{eqn:MD.flow}. This suggests two ways to generalize the
method. Formally, we may replace the Bregman divergence in
\eqref{eqn:MD.update.equiv} by a $\lambda$-logarithmic divergence. While the
resulting proximal method is well-defined, unfortunately the first-order
condition cannot be solved explicitly to yield a simple update as in mirror
descent (see \eqref{eqn:MD.update}). We study instead the continuous time
Riemannian gradient flow with respect to the metric $G_{\lambda}$, and it turns
out that this is much more tractable. We fix $\lambda \neq 0$ and let a regular
$c_{\lambda}$-convex generator $\varphi$ be given on the convex domain $\Theta$.
\begin{definition}[Conformal mirror descent in continuous time] \label{def:cmd}
Let $f: \Theta \rightarrow \mathbbm{R}$ be a differentiable function to be
minimized. Given an initial value $\theta_0 \in \Theta$, the continuous time
conformal mirror descent is the Riemannian gradient flow given by
\begin{equation} \label{eqn:L.lambda.flow}
\dv{t} \theta_t = -\mathrm{grad}_{G_{\lambda}} f(\theta_t),
\end{equation}
where $\mathrm{grad}_{G_{\lambda}} f = G_{\lambda}^{-1} \nabla_{\theta}
f$ is the Riemannian gradient with respect to the metric $G_{\lambda}$ in
\eqref{eqn:L.alpha.metric}.
\end{definition}

While any Riemannian metric $G(\theta)$ can be used to define a gradient flow,
implementation of the flow generally requires computation of the Riemannian
gradient $G^{-1} \nabla_{\theta} f$. In \eqref{eqn:MD.flow}, the mirror map
$\nabla \phi$ eliminates the need to compute $G_0^{-1}$ because $G_0 = \nabla^2
\phi$ is the Jacobian of the mirror map. Here, we show that a similar property
holds under the $\lambda$-mirror map. We let $I_d$ denote the $d \times d$
identity matrix.

\begin{theorem}[Dynamics under the $\lambda$-mirror map]  \label{thm:flow.main1}
Consider the flow \eqref{eqn:L.lambda.flow}. Let $\eta_t
= \nabla^{(\lambda)}_{\theta}\varphi(\theta_t)$ be the dual variable under the
$\lambda$-mirror map \eqref{eqn:L.alpha.mirror.map}. Then
\begin{equation} \label{eqn:L.alpha.flow.alternative}
\begin{split}
\dv{}{t}\eta_t = -\Pi_{\lambda}(\theta_t) (I_d + \lambda \eta_t \theta_t^{\top})
\nabla_{\theta} f(\theta_t),
\end{split}
\end{equation}
where $\Pi_{\lambda}(\theta) \defeq 1 + \lambda
\DP{\theta}{\nabla^{(\lambda)}_{\p}\varphi(\p)} = 1 + \lambda\DP{\p}{\d}$.
\end{theorem}
\begin{proof}
Under the primal coordinate system, the metric $G_{\lambda}$ is given by
\begin{equation*}
\begin{split}
(G_{\lambda}(\theta))_{ij} &= - \left. \frac{\partial^2}{\partial \theta^i
  \partial \theta^{'j}}{\bf L}_{\lambda, \varphi}[\theta : \theta']
  \right|_{\theta = \theta'} \\
&= - \left.  \frac{\partial^2}{\partial \theta^i \partial \theta^{'j}}  \left\{
  \varphi(\theta) + \psi(\eta') - \frac{1}{\lambda} \log (1 + \lambda
  \DP{\theta}{\eta'}) \right\} \right|_{\theta = \theta'}\\
&= \frac{1}{\Pi_{\lambda}(\theta)} \left\{  \frac{\partial \eta^i}{\partial
  \theta^j} -  \frac{\lambda}{\Pi_{\lambda}(\theta)} \eta^i \sum_{k = 1}^d
  \theta^k \frac{\partial \eta^k}{\partial \theta^j} \right\}.
\end{split}
\end{equation*}
In matrix form, we have
\begin{equation} \label{eqn:G.lambda.matrix}
G_{\lambda}(\theta) = \frac{1}{\Pi_{\lambda}(\theta)} \left( I_d
- \frac{\lambda}{\Pi_{\lambda}(\theta)} \eta \theta^{\top} \right)
\frac{\partial \eta}{\partial \theta}(\theta),
\end{equation}
where $\frac{\partial \eta}{\partial \theta} = \left( \frac{\partial
\eta^i}{\partial \theta^j} \right)$ is the Jacobian of the transformation
$\theta \mapsto \eta$. Now we may invert \eqref{eqn:G.lambda.matrix} using the
Sherman–Morrison formula to get
\begin{equation*}
G_{\lambda}^{-1}(\theta) = \Pi_{\lambda}(\theta) \frac{\partial \theta}{\partial
\eta}(\eta) (I_d + \lambda \eta \theta^{\top}).
\end{equation*}

By definition, the gradient flow \eqref{eqn:L.lambda.flow} is given by
\[
\dv{t} \theta_t = -G_{\lambda}^{-1}(\theta_t) \nabla_{\theta} f(\theta_t).
\]
Expressing the flow in terms of the dual variable, we have, by the chain rule
again,
\begin{equation*}
\begin{split}
\dv{t} \eta_t &= \frac{\partial \eta}{\partial \theta}(\theta_t) \dv{t} \theta_t
\\
  &= - \frac{\partial \eta}{\partial \theta}(\theta_t)  \Pi_{\lambda}(\theta_t)
  \frac{\partial \theta}{\partial \eta}(\eta_t) (I_d + \lambda \eta_t
  \theta_t^{\top})  \nabla_{\theta} f(\theta_t) \\
  &= - \Pi_{\lambda}(\theta_t)  (I_d + \lambda \theta_t \eta_t^{\top})
  \nabla_{\theta} f(\theta_t).
\end{split}
\end{equation*}
\end{proof}

Next, by using the fact that $G_{\lambda}$ is a conformal Hessian metric, we
show that the conformal mirror descent gradient flow can be viewed as a time
change of a Hessian gradient flow.

\begin{theorem} [Time-change of Hessian gradient flow]  \label{thm:flow.main2}
Let $(\tilde{\theta}_s)_{s \geq 0}$ be the Hessian gradient flow
\eqref{eqn:MD.flow} with respect to the Bregman generator $\Phi_{\lambda}
= \frac{1}{\lambda}(e^{\lambda \varphi} - 1)$. Consider the time change $s
= s_t$, where $\dv{t} s_t =  \exp(\lambda \varphi(\tilde{\theta}_{s_t}))$. Then
$\theta_t = \tilde{\theta}_{s_t}$ is the conformal mirror descent
\eqref{eqn:L.lambda.flow} induced by $\varphi$. In particular, let $\zeta_t
= \nabla \Phi_{\lambda}(\theta_t)$ be the dual variable with respect to the
Bregman generator $\Phi_{\lambda}$. Then the flow can be expressed as $\dv{t}
\zeta_t = - e^{\lambda \varphi(\theta_t)} \nabla_{\theta} f(\theta_t)$.
\end{theorem}
\begin{proof}
By \eqref{eqn:L.lambda.flow} and \eqref{eqn:L.alpha.metric}, we have
\begin{equation} \label{eqn:L.alpha.flow.conformal}
\dv{t} \theta_t = - G_{\lambda}^{-1}(\theta_t) \nabla_{\theta} f(\theta_t)
= - e^{\lambda \varphi(\theta_t)} \tilde{G}_0^{-1}(\theta_t) \nabla_{\theta}
f(\theta_t),
\end{equation}
where $\tilde{G}_0 = \nabla_{\theta}^2 \Phi_{\lambda}$. Let $\tilde{\theta}(s)$
be the Hessian gradient flow \eqref{eqn:MD.flow} induced by the metric
$\tilde{G}_0$, and let $s = s_t$ be the given time change. Applying the chain
rule in \eqref{eqn:MD.flow}, we have
\begin{equation*}
\begin{split}
\dv{t} \tilde{\theta}_{s_t} &= \dv{s} \tilde{\theta}_{s_t}
\dv{t} s_t = - \tilde{G}_0^{-1}(\tilde{\theta}_{s_t})
\nabla_{\theta} f(\tilde{\theta}_{s_t}) \dv{t} s_t \\
 &= - e^{\lambda \varphi(\tilde{\theta}_{s_t})}
 \tilde{G}_0^{-1}(\tilde{\theta}_{s_t}) \nabla_{\theta} f(\tilde{\theta}_{s_t}).
\end{split}
\end{equation*}
Comparing this with \eqref{eqn:L.alpha.flow.conformal}, we see that
$\tilde{\theta}_{s_t} = \theta_t$. The proof of the last statement is similar.
\end{proof}

By Theorem \ref{thm:flow.main2}, the trajectory of a conformal mirror descent
gradient flow is the same as that of a Hessian gradient flow; the conformal
transformation of the metric introduces a {\it time-varying learning rate}
depending on the value $\varphi(\theta_t)$. The main novelty of conformal mirror
descent is that the $\lambda$-duality suggests novel choices of the generator
$\varphi$; additionally, the $\lambda$-mirror map is more natural in certain
problems. For example, in Section
\ref{sub:online_parameter_estimation_for_generalized_exponential_family} we
apply it to online natural gradient learning for some generalized exponential
families.

\medskip
We give a concrete example which generalizes \cite[Theorem 5.5]{pw18}. For
a given regular $c_{\lambda}$-convex generator $\varphi$, consider minimizing
either $f(\theta) = {\bf L}_{\lambda, \varphi}[\theta^* : \theta]$ or $f(\theta)
= {\bf L}_{\lambda, \varphi}[\theta : \theta^*]$ for some $\theta^* \in \Theta$.
Note that $f$ is typically not convex in $\theta$ (or $\eta$). We show that the
conformal mirror descent evolves along geodesics of the underlying dualistic
structure. Also, see \cite{fujiwara1995gradient} for a detailed analysis of the
dually flat case.

\begin{proposition} [Primal and dual flows] \label{prop:primal.dual.flows}
\begin{enumerate}
    \item[(i)] The trajectory of the primal flow
    \begin{equation} \label{eqn:primal.flow}
    \dv{}{t} \theta_t = -\mathrm{grad}_{G_{\lambda}}
    \LD{\theta^*}{\cdot}(\theta_t)
    \end{equation}
    follows a time-changed primal geodesic, i.e., along the straight line from
    $\theta_0$ to $\theta^*$ under the primal coordinate system.
    \item[(ii)] The trajectory of the dual flow
    \begin{equation} \label{eqn:dual.flow}
    \dv{}{t} \theta_t = -\mathrm{grad}_{G_{\lambda}}
    \LD{\cdot}{\theta^*}(\theta_t)
    \end{equation}
    follows a time-changed dual geodesic, i.e., along the straight line from
    $\eta_0$ to $\eta^*$ under the dual coordinate system.
\end{enumerate}
\end{proposition}
\begin{proof}
We first consider the dual flow \eqref{eqn:dual.flow}. Using the self-dual
representation \eqref{eqn:log.self.dual},
\[
\nabla_{\theta} {\bf L}_{\lambda, \varphi}[\cdot : \theta^*] = \nabla_{\theta}
\varphi(\theta) - \frac{\eta^*}{1 + \lambda \DP{\theta}{\eta^*}} = \frac{\eta}{1
+ \lambda \DP{\theta}{\eta}} - \frac{\eta^*}{1 + \lambda \DP{\theta}{\eta^*}},
\]
where the last equality can be verified using the definition
\eqref{eqn:L.alpha.mirror.map} of the $\lambda$-mirror map.

By Theorem \ref{thm:flow.main1} we have, after some simplification,
\begin{equation} \label{eqn:dual.flow.update}
\dv{}{t} \eta_t = - \frac{1 + \lambda \DP{\theta_t}{\eta_t}}{1 + \lambda
\DP{\theta_t}{\eta^*}} (\eta_t - \eta^*).
\end{equation}
Thus, the dual flow evolves along a time-changed dual geodesic.

Since ${\bf L}_{\lambda, \varphi}[\theta^* : \theta] = {\bf L}_{\lambda,
\psi}[\eta : \eta^*]$ and both ${\bf L}_{\lambda, \varphi}$ and ${\bf
L}_{\lambda, \psi}$ induce the same Riemannian metric, the primal flow
\eqref{eqn:primal.flow} for ${\bf L}_{\lambda, \varphi}$ is equivalent to the
dual flow for ${\bf L}_{\lambda, \psi}$. By the case proved above, we have that
the trajectory follows a time-changed straight line under the
$\theta$-coordinates.
\end{proof}

To implement conformal mirror descent in practice, the flow
\eqref{eqn:L.lambda.flow} must be discretized. From Definition \ref{def:cmd} and
Theorems \ref{thm:flow.main1} and \ref{thm:flow.main2}, we have the following
three forward Euler discretizations:
\begin{itemize}
    \item Primal Euler discretization: $\theta_{k+1} = \theta_k - \delta
      G_{\lambda}^{-1}(\theta_k) \nabla_{\theta} f(\theta_k)$.
    \item Dual Euler discretization: $\eta_{k+1} = \eta_k - \delta
      \Pi_\lambda(\theta_k)\left( I_d + \lambda \theta_{k}\eta_k^{\top}
      \right)\nabla_{\theta} f(\theta_k)$.
    \item Mirror descent with adaptive learning rate: $\zeta_{k+1} = \zeta_k
      - \delta e^{\lambda \varphi(\theta_k)} \nabla_{\theta} f(\theta_k)$, where
      $\zeta_k = \nabla_{\theta} \Phi_{\lambda}(\theta_k)$.
\end{itemize}
A detailed analysis of these (and possibly other) discretization schemes is
beyond the scope of this paper. In the next subsection, we study the convergence
of conformal mirror descent in continuous time.

\subsection{Convergence results} \label{sec:flow.convergence.statements}
In this subsection, we present continuous time convergence results for
conformal mirror descent that are analogous to those of mirror descent.
Our main tool is Lyapunov analysis following \cite{wilson18}. In what follows,
we let $(\theta_t)_{t \geq 0}$ be the solution to the gradient flow
\eqref{eqn:L.lambda.flow} for a given continuously differentiable and convex
function $f: \Theta \rightarrow \mathbbm{R}$. We also let $\theta^*$ be
a minimizer of $f$ over $\Theta$.

We first observe that the logarithmic divergence is a Lyapunov function of the
gradient flow.

\begin{lemma} \label{lem:L.div.Lyapunov}
The functional $\mathcal{E}_t = {\bf L}_{\lambda, \varphi}[\theta^* : \theta_t]$
is a Lyapunov function of the gradient flow, i.e., $\dv{}{t} \mathcal{E}_t \leq
0$.
\end{lemma}
\begin{proof}
Using the self-dual representation \eqref{eqn:log.self.dual}, we have
\begin{equation*}
\begin{split}
\dv{}{t} \mathcal{E}_t = \dv{}{t} \left( \varphi(\theta^*) + \psi(\eta_t)
- \frac{1}{\lambda} \log (1 + \lambda \langle \theta^*, \eta_t \rangle) \right)
= \frac{\DP{\theta_t}{\dv{}{t} \eta_t}}{1 + \lambda \DP{\theta_t}{\eta_t}}
- \frac{\DP{\theta^*}{\dv{}{t} \eta_t}}{1 + \lambda \DP{\theta^*}{\eta_t}}.
\end{split}
\end{equation*}
Using \eqref{eqn:L.alpha.flow.alternative} and simplifying, we have
\begin{equation} \label{eqn:Lyapunov.1}
\dv{}{t}
\mathcal{E}_t = \frac{1 + \lambda \DP{\theta_t}{\eta_t}}{1 + \lambda
\DP{\theta^*}{\eta_t}} \DP{\nabla_{\theta} f(\theta_t)}{\theta^* - \theta_t}
\leq \frac{1 + \lambda \DP{\theta_t}{\eta_t}}{1 + \lambda \DP{\theta^*}{\eta_t}}
(f(\theta_t) - f(\theta^*)) \leq 0.
\end{equation}
\end{proof}

\begin{theorem} \label{thm:convergence.main.text}
  Define $\tau_t = \int_0^t e^{\lambda \varphi(\theta_s)} \dd{s}$, so that
$\dot{\tau}_t = \dv{}{t} \tau_t = e^{\lambda \varphi(\theta_t)}$. Let
$ \hat{\theta}_t = \frac{1}{\tau_t} \int_0^t \theta_s \dot{\tau}_s \dd{s}$,
which is a weighted average of the trajectory up to time $t$. If $\theta^*$ is
a minimizer of $f$ over $\Theta$, then
\begin{equation} \label{eqn:L.lambda.conv.bound}
f(\hat{\theta}_t) - f(\theta^*) \leq \frac{{\bf B}_{\Phi_{\lambda}}[\theta^*
: \theta_0]}{\tau_t},
\end{equation}
where $\Phi_{\lambda} = \frac{1}{\lambda} (e^{\lambda \varphi} - 1)$ is the
Bregman generator. In particular, if $f$ is strictly convex, then
$f(\hat{\theta}_t) - f(\theta^*) = O(\frac{1}{t})$ as $t \rightarrow \infty$.
\end{theorem}
\begin{proof}
This result can be derived using Theorem \ref{thm:flow.main2} and convergence
results of Hessian gradient flow (see e.g.~\cite[Section 2.1.3]{w18}). For
completeness, we give a self-contained proof. Using a similar argument as in the
proof of Lemma \ref{lem:L.div.Lyapunov}, we have that
\[
\mathcal{E}_t = \frac{1}{\lambda} \left(1 -  e^{-\lambda {\bf L}_{\lambda,
\varphi}[\theta^* : \theta_t]}\right) + \int_0^t e^{\lambda (\varphi(\theta_s)
- \varphi(\theta^*))} (f(\theta_s) - f(\theta^*)) \dd{s}
\]
satisfies
\begin{equation} \label{eqn:E.derivative}
\dv{}{t} \mathcal{E}_t = - e^{-\lambda (\varphi(\theta_t) - \varphi(\theta^*))}
{\bf B}_{\Phi_{\lambda}}[\theta^* : \theta_t] \leq 0,
\end{equation}
and hence is another Lyapunov function. Since $\mathcal{E}_t$ is non-increasing,
we have
\begin{equation} \label{eqn:Lyapunov.computation}
e^{-\varphi(\theta^*)} \tau_t \int_0^t \frac{\dot{\tau}_s}{\tau_t} (f(\theta_s)
- f(\theta^*)) \dd{s} \leq \mathcal{E}_t \leq \mathcal{E}_0
= \frac{1}{\lambda} (1 - e^{-\lambda {\bf L}_{\lambda, \varphi}[\theta^*
: \theta_0]}).
\end{equation}
Note that by \eqref{eqn:conformal.divergence}, the last expression in
\eqref{eqn:Lyapunov.computation} is equal to $e^{-\varphi(\theta^*)} {\bf
B}_{\Phi_{\lambda}}[\theta^* : \theta_0]$. Since $f(\cdot) - f(\theta^*)$ is
convex, by Jensen's inequality we have
\[
f(\hat{\theta}_t) - f(\theta^*) \leq \int_0^t \frac{\dot{\tau}_s}{\tau_t}
(f(\theta_s) - f(\theta^*)) \dd{s} \leq \frac{1}{\tau_t}{\bf
B}_{\Phi_{\lambda}}[\theta^* : \theta_0].
\]
If $f$ is strictly convex, from \eqref{eqn:Lyapunov.1} we have that $\lim_{t
\rightarrow \infty} \theta_t = \theta^*$. Since $e^{\lambda \varphi(\theta_t)}
\rightarrow e^{\lambda \varphi(\theta^*)}$, the quantity $\tau_t = \int_0^t
e^{\lambda \varphi(\theta_s)} \dd{s}$ grows linearly as $t \rightarrow \infty$.
It follows from \eqref{eqn:L.lambda.conv.bound} that $f(\hat{\theta}_t)
- f(\theta^*) = O(\frac{1}{t})$ as $t \rightarrow \infty$.
\end{proof}

\section{Online estimation of generalized exponential family}
\label{sub:online_parameter_estimation_for_generalized_exponential_family}
Mirror descent is often used to estimate parameters of
stochastic models, both offline and online. Using a duality between the exponential
family and Bregman divergence \cite{bmdg05}, the authors of \cite{rm15}
considered online parameter estimation for exponential families, and showed that
the mirror descent step is equivalent to the natural gradient step \cite{a98}.
In this section we generalize this result to the {\it $\lambda$-exponential
family} introduced in \cite{WZ21}.

We begin with some preliminaries. Following \cite{WZ21}, by a {\it
$\lambda$-exponential family} we mean a parameterized probability density (with
respect to a reference measure $\nu$) of the form
\begin{equation} \label{eqn:F.alpha.family}
p_{\theta}(x) = (1 + \lambda \langle \theta , F(x) \rangle)_+^{1/\lambda}
e^{-\varphi(\theta)},
\end{equation}
where $x_+ = \max\{x, 0\}$ and $F(x) = (F^1(x), \ldots, F^d(x))$ is a vector of
statistics. For example, if $\nu$ is the Lebesgue measure on $\R$, $\lambda \in
(-2, 0)$ and $F(x) = (x, x^2)$, then we obtain from \eqref{eqn:F.alpha.family}
the Student's $t$ distribution (as a location-scale family) with
$\frac{-2}{\lambda} - 1 > 0$ degrees of freedom (see Example \ref{eg:student-t}
below). The density \eqref{eqn:F.alpha.family} is a generalized or deformed
exponential family and is a reparameterized version of the {\it $q$-exponential
family} (where $q = 1 - \lambda$) in statistical physics (see \cite[Section
3]{WZ21} for the precise relation).\footnote{Note that parameterized densities
similar to \eqref{eqn:F.alpha.family} were studied by other authors such as
\cite{km20}, but their motivations were quite different.} As $\lambda
\rightarrow 0$, we recover the usual exponential family. Under suitable
regularity conditions (including $\lambda < 1$ or $q = 1 - \lambda > 0$), it can
be shown that the divisive normalization function $\varphi$ in
\eqref{eqn:F.alpha.family} is $c_{\lambda}$-convex on the parameter space
$\Theta$ and hence, defines a $\lambda$-logarithmic divergence. This divergence
can be interpreted probabilistically as ${\bf L}_{\lambda, \varphi}[ \theta
: \theta'] = \HHr_q ( p_{\theta'} || p_{\theta})$, where $\HHr_q$ is the {\it
R\'{e}nyi divergence}  of order $q$. Consequently, the induced Riemannian metric
is a constant multiple of the {\it Fisher information metric} $\mathcal{I}$
\cite{vh14}:
\begin{equation} \label{eqn:Fisher}
G_{\lambda}(\theta) = (1 - \lambda) \mathcal{I}(\theta).
\end{equation}
Moreover, the dual variable $\eta = \nabla_{\theta}^{(\lambda)} \varphi(\theta)$
under the $\lambda$-mirror map can be interpreted as a generalized expectation
parameter known as the {\it escort expectation}:
\[
  \eta = \int F(x) \frac{p_{\theta}(x)^q}{ \int p_{\theta}^{q} \dd{\nu}}
  \dd{\nu}(x).
\]
In fact, the density \eqref{eqn:F.alpha.family} maximizes the R\'{e}nyi entropy
of order $q$ subject to constraints on the escort expectation. These (and other)
results nicely parallel those of the exponential family.

Consider now online estimation of \eqref{eqn:F.alpha.family} under
i.i.d.~sampling. By considering the distribution of $Y = F(X)$, we consider
a $\lambda$-exponential family on $\R^d$ as
\begin{equation} \label{eqn:F.alpha.y}
p_{\theta}(y) = (1 + \lambda \DP{\theta}{y})_+^{1/\lambda} e^{-\varphi(\theta)}.
\end{equation}
Suppose we observe data points $y_k$, $k = 1, 2, \ldots$. Let the current guess
of the parameter be $\theta_k$. After observing $y_k$, we update $\theta_k$ to
$\theta_{k+1}$ by a minimizing gradient step with respect to the log-loss
\begin{equation} \label{eqn:log.loss}
f_k(\theta) = - \log p_{\theta}(y_k) = \varphi(\theta) - \frac{1}{\lambda}
\log(1 + \lambda \DP{\theta}{y_k}).
\end{equation}
Note that the negative log-likelihood $f_k$ is typically not convex in $\theta$.
We do this by discretizing the conformal mirror descent
\eqref{eqn:L.alpha.flow.alternative}, where the generating function $\varphi$ is
the potential function in \eqref{eqn:F.alpha.y}. Since $G_{\lambda}$ is
a multiple of the Fisher metric, the forward Euler step of
\eqref{eqn:L.alpha.flow.alternative} in dual coordinates leads to the
(unconstrained) {\it natural gradient update}
\begin{equation} \label{eqn:online.L.alpha.step}
\eta_{k+1} = \eta_k - \delta_k \Pi_{\lambda}(\theta_k) (I_d + \lambda \eta_k
\theta_k^{\top}) \nabla_{\theta} f_k(\theta_k),
\end{equation}
where $\delta_k > 0$ is the learning rate. Simplifying
\eqref{eqn:online.L.alpha.step}, we obtain an explicit and clean update which is
not obvious from the time change perspective.

\begin{theorem} [Online natural gradient step for $\lambda$-exponential family]
\label{thm:online.gradient.step}
The online natural gradient update \eqref{eqn:online.L.alpha.step} is given by
\begin{equation} \label{eqn:F.lambda.online.step}
\eta_{k+1} = \eta_k + \delta \frac{1 + \lambda \DP{\theta_k}{\eta_k}}{1
+ \lambda \DP{\theta_k}{y_k} } (y_k - \eta_k).
\end{equation}
\end{theorem}
\begin{proof}
Differentiating $f_k(\theta)$ in \eqref{eqn:log.loss}, we have
\[
\nabla_{\theta}f(\theta) = \frac{\eta}{1 + \lambda \DP{\theta}{\eta}}
- \frac{y_k}{1 + \lambda \DP{\theta}{y_k}},
\]
which has the same form as in the dual gradient flow in Proposition
\ref{prop:primal.dual.flows}(ii). (This is not a coincidence in view of the
duality between $\lambda$-exponential family and $\lambda$-logarithmic
divergence; see \cite[Section VI]{WZ21}.) Continuing the computation as in the
proof of Proposition \ref{prop:primal.dual.flows}, we obtain
\eqref{eqn:F.lambda.online.step} which is the discrete analogue of
\eqref{eqn:dual.flow.update}.
\end{proof}

Since \eqref{eqn:online.L.alpha.step} is a natural gradient update, by
\cite[Theorem 2]{a98} the algorithm (when $\delta_k = \frac{1}{k}$) is Fisher
efficient as $k \rightarrow \infty$. When $\lambda \rightarrow 0$, we recover
the linear update for exponential families derived in \cite{rm15}. In general,
an extra projection step, which is also necessary for the exponential family
($\lambda = 0$), is needed to retain $\theta_{k+1}$ in the parameter set
$\Theta$ (or $\eta_{k+1}$ in $H$). The adjustment is implemented in the
experiments below by a reflection across the boundary.

\begin{figure}[!t]
\centering
\includegraphics[width=2in]{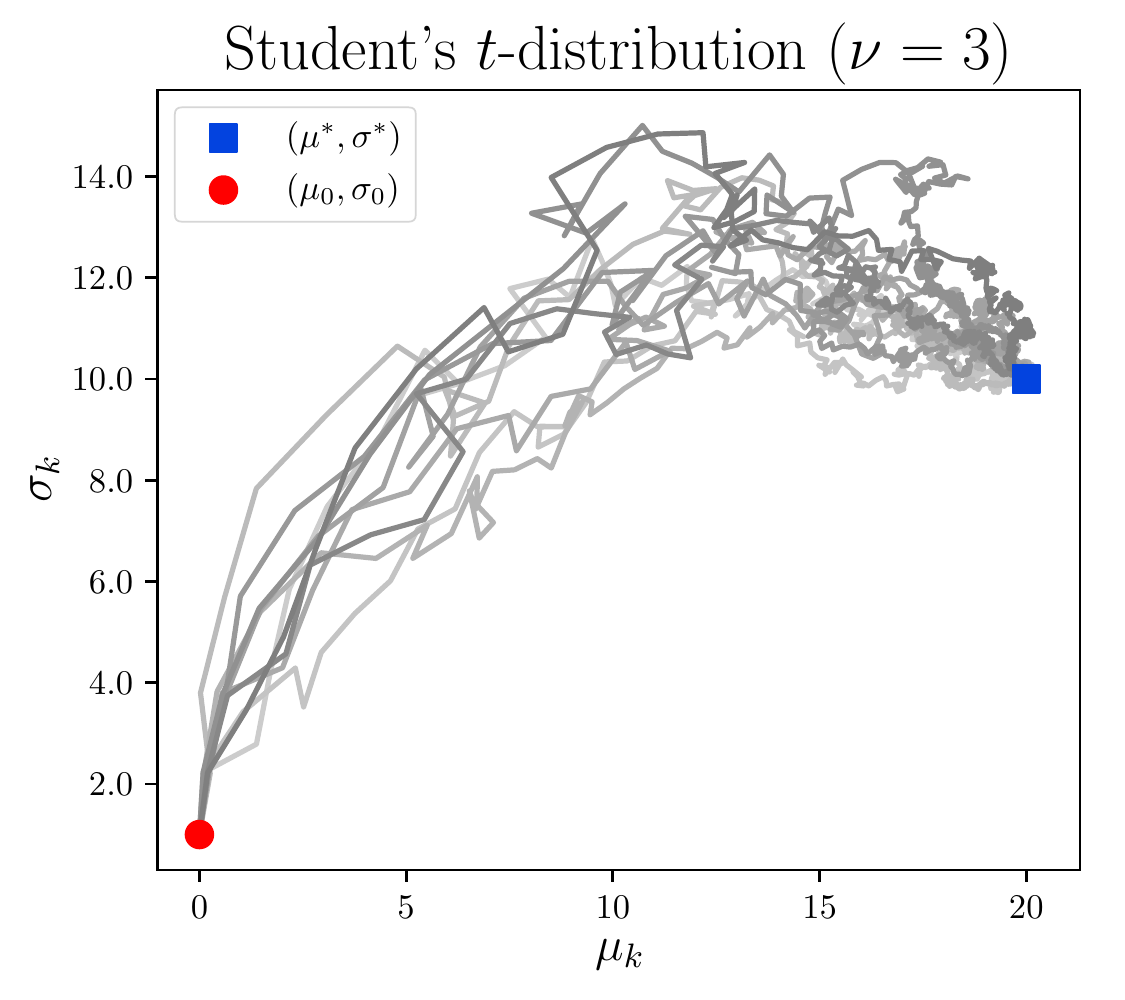}
\includegraphics[width=2in]{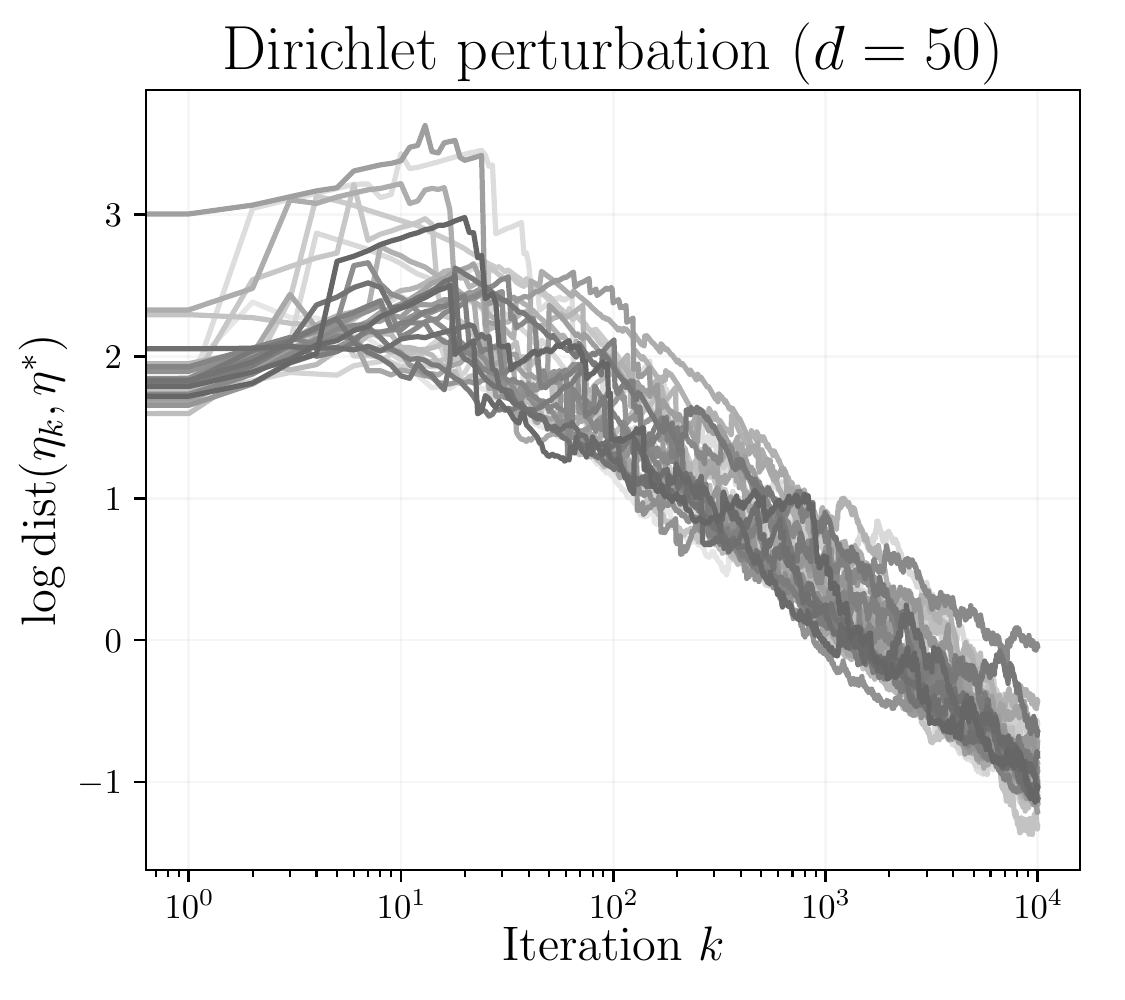}
\caption{{\bf Left:} $10$ trajectories of \eqref{eqn:F.lambda.online.step} for
the Student $t$-distribution location scale family \eqref{eqn:student-t} with
$\nu = 3$ degrees of freedom, where the learning rate is $\delta_k = 1/k$. We
show the dynamics of $(\mu_k, \sigma_k)$ over a sample of $10000$ data points.
Red dot: initial guess $(\mu_0, \sigma_0)$. Blue square: true parameter $(\mu^*,
\sigma^*)$. {\bf Right:} Plot of $\log \mathrm{dist}(\eta_{k}, \eta^{*})$
against $\log_{10} k$ for the Dirichlet perturbation model
(Example~\ref{eg:dir.perturbation.maintext}) over $30$ trajectories
of~\eqref{eqn:F.lambda.online.step}, each with $10000$ data points. Here
$\mathrm{dist}(\eta, \eta') = |\log (\eta) - \log (\eta')|$ is used as a metric
on the dual domain $H = (0, \infty)^d$. In this simulation $d = 50$, $\lambda
= -0.3$ and $\delta_k = 1/k$. We observe that $\mathrm{dist}(\eta_k, \eta^{*})$
decays like $O(k^{-1/2})$ which is consistent with the asymptotic efficiency of
online natural gradient learning.}
\label{fig:online-parest}
\end{figure}

\begin{example}[Student's $t$-distribution as a location-scale
  family]\label{eg:student-t}
Let $\nu > 0$ be a constant. The Student's $t$-distribution with $\nu$ degrees
of freedom has density on $\R$ given by
\begin{equation}\label{eqn:student-t}
p(x ;\mu,\sigma, \nu) = \frac{\Gamma((\nu + 1)/2)}{\Gamma(\nu/2) \sqrt{\nu \pi}
\sigma}\left( 1 + \frac{1}{\nu} \frac{(x-\mu)^2}{\sigma^2} \right)^{-(\nu+1)/2},
\end{equation}
where $\mu$ and $\sigma$ are the location and scale parameters, respectively,
and $\Gamma$ is the gamma function.\footnote{Here the dominating measure is the
Lebesgue measure on $\R$ and $\nu \in (0, \infty)$ denotes the degrees of
freedom.} In the following, we regard $\nu$ as known and consider online
estimation of $(\mu, \sigma)$.

Let $\lambda = \tfrac{-2}{\nu + 1} \in (-2, 0)$ and $F(x) = (x, x^2)^{\top}$.
Then we can express \eqref{eqn:student-t} as a $\lambda$-exponential family
$p_{\theta}(x) = (1 + \lambda \DP{\theta}{F(x)})^{1/\lambda}
e^{-\varphi(\theta)}$. The natural parameter $\theta$ is given by
\begin{equation*}
\theta = (\theta^1, \theta^2) = \left( \dfrac{2 \mu}{- \lambda \mu^{2}
+ \sigma^{2} \left(\lambda + 2\right)}, -\dfrac{1}{- \lambda \mu^{2}
+ \sigma^{2} \left(\lambda + 2\right)}\right),
\end{equation*}
and takes values in the set
\begin{equation*}
\Theta = \left\{\theta = (\theta^1,\theta^2) \in \R^2 : \theta^2 < 0 \text{ and
} \lambda(\theta^1)^2 - 4\theta^2 > 0\right\}.
  \end{equation*}
The potential function $\varphi$ is given on $\Theta$ by
\begin{equation*}
\varphi(\theta) = \log{\left( \frac{\sqrt{\frac{\lambda (\theta^1)^{2}
- 4 \theta^{2}}{\lambda + 2}}}{-2 \theta^{2}} \right)}
- \frac{1}{\lambda}\log{\left( \frac{-4 \theta^{2}}{\lambda (\theta^1)^{2}
- 4 \theta^{2}} \right)} + C,
  \end{equation*}
where $C$ is a constant depending only on $\nu$. By a straightforward
computation, we obtain explicit expressions of the $\lambda$-mirror map and its
inverse:
\begin{align*}
\eta &= \nabla_{\theta}^{(\lambda)}\varphi_{\lambda}(\theta) = \left(
\frac{-\theta^{1}}{2 \theta^{2}}, \frac{\lambda (\theta^{1})^{2}
+ (\theta^1)^{2} - 2 \theta^{2}}{2 (\lambda + 2) \theta_{2}^{2} } \right),\\
\theta &= \nabla_{\eta}^{(\lambda)}\psi(\eta) = \left( \frac{-2 \eta^{1}}{2
(\lambda + 1) (\eta^1)^2 - (\lambda + 2) \eta^{2}} , \frac{1}{2 (\lambda + 1)
(\eta^1)^2 - (\lambda + 2) \eta^{2}}\right).
  \end{align*}
In Figure \ref{fig:online-parest} (left), we show ten trajectories (in terms of
$(\mu_k, \sigma_k)$) of the algorithm \eqref{eqn:F.lambda.online.step} with
$\delta_k = 1/k$, where the true parameter is $(\mu^*, \sigma^*)$ and the
initial guess is $(\mu_0, \sigma_0)$. As expected, the iterates converge to
$(\mu^*, \sigma^*)$ as $k \rightarrow \infty$. The preceding computations can be
generalized to the multivariate location-scale $t$-distribution where the
degrees of freedom is also assumed to be known.
\end{example}

\begin{example} [Dirichlet perturbation on the unit simplex]
  \label{eg:dir.perturbation.maintext}
The Dirichlet perturbation model is a fundamental example of the
$\lambda$-exponential family (see \cite[Example 3.14]{WZ21}) and is closely
related to the {\it Dirichlet optimal transport problem} studied in \cite{pw16,
pw18, pw18b}; see also Section \ref{sub:Dirichlet} below, where we use the
Dirichlet transport to define gradient flows on the simplex. Fix $d \geq 1$ and
consider the open unit simplex
\begin{equation} \label{eqn:simplex}
\triangle^{1+d} = \left\{ p = (p^0, p^1, \ldots, p^d) \in (0, 1)^{1+d}: \sum_{i
= 0}^d p^i = 1 \right\}.
\end{equation}
Given $p, q \in \triangle^{1 + d}$, define the {\it perturbation operation} by
\begin{equation} \label{eqn:perturbation}
p \oplus q = \left( \frac{p^0 q^0}{\sum_{j=0}^d p^j q^j}, \ldots, \frac{p^d
q^d}{\sum_{j=0}^d p^j q^j} \right).
\end{equation}
This is the vector addition operation under the {\it Aitchison geometry} in
compositional data analysis \cite{EPMB03}. Let $\sigma > 0$ and let $\lambda
= -\sigma < 0$. Fix $p \in \triangle^{1 + d}$, which we regard as the unknown
parameter, and let $D = (D^0, \ldots, D^d)$ be a random vector whose
distribution is the Dirichlet distribution with parameters $(\sigma^{-1}/(1
+ d), \ldots, \sigma^{-1}/(1 + d)) \in (0, \infty)^{1+d}$. As $\sigma
\rightarrow 0$, the distribution of $D$ converges weakly to the point mass at
the barycenter $(1/(1 + d), \ldots, 1/(1 + d))$. Thus, we may regard $\sigma$ as
a noise parameter. The Dirichlet perturbation model is specified as
\begin{equation} \label{eqn:Dirichlet.perturbation}
Q = p \oplus D.
\end{equation}
It may be regarded as a multiplicative analogue of the Gaussian additive model
$Y = X + \epsilon$ where $\epsilon \sim N(0, \sigma^2 I_d)$.

By \cite[Proposition 3.16]{WZ21}, the distribution of $Q$ can be
expressed as a $\lambda$-exponential family with $\lambda = -\sigma < 0$, if we
let $F(q) = (q^1/q^0, \ldots, q^d/q^0)$ and $\theta = (p^0/\lambda p^1, \ldots,
p^0/\lambda p^d) \in \Theta = (-\infty, 0)^d$. By \cite[(III.30)]{WZ21}, the
potential function is given by
\[
\varphi(\theta) = \frac{1}{\lambda(1 + d)} \sum_{i = 1}^d \log (-\theta^i).
\]
Letting $d = 1$ and $\lambda = -1$ (and replacing $\theta$ by $-\theta$),
recovers the first example in Table \ref{tab:examples}. The dual variable $\eta$
is given by $\eta^i = \frac{1}{\lambda \theta^i} = \frac{p^i}{p^0}$. In
Figure~\ref{fig:online-parest} (right), we illustrate the $O(k^{-1/2})$
convergence rate of the online estimation algorithm
\eqref{eqn:F.lambda.online.step}. In fact, it can be verified that the update
\eqref{eqn:F.lambda.online.step}, when expressed in terms of $p_k$ (the current
estimate of $p$) and $q_{k}$ (the new data point) with values in $\triangle^{1
+ d}$, is independent of the value of $\lambda < 0$. Thus, for online estimation
of the Dirichlet perturbation model, we may treat $\sigma > 0$ (or $\lambda
< 0$) as {\it unknown}.
\end{example}

\section{Gradient flows on the simplex via Dirichlet transport}
\label{sub:Dirichlet}
By Brenier's theorem \cite{b91}, the mirror map $\zeta = \nabla \phi(\theta)$ in
classical (Bregman) mirror descent can be interpreted as an optimal transport
map for the quadratic cost $c(x, y) = \frac{1}{2}|x - y|^2$. Also, the Bregman
divergence is the $c$-divergence of the quadratic cost. This suggests an
interpretation of mirror descent in terms of optimal transport. In fact, our
conformal mirror descent generalizes this set-up to the logarithmic cost
$c_{\lambda}(x, y) = \frac{-1}{\lambda} \log (1 + \lambda \DP{x}{y})$ for
$\lambda \neq 0$. In this section, we specialize to the unit simplex and the
case $\lambda = -1$. Using the {\it Dirichlet optimal transport problem} studied
in \cite{pw18b}, we define a family of gradient flows on the unit simplex and
compare them with the entropic descent, which is an important example of mirror
descent.

\subsection{Dirichlet transport}
Following \cite{pw18b}, we define the {\it Dirichlet cost function} on
$\triangle^{n} \times \triangle^{n}$ (where $n = 1 + d \geq 2$) by
\begin{equation} \label{eqn:Dirichlet.cost}
c(p, q) = \log \left(  \sum_{i = 0}^{n-1} \frac{1}{n} \frac{q^i}{p^i} \right)
- \sum_{i = 0}^{n-1} \frac{1}{n} \log \frac{q^i}{p^i}.
\end{equation}
It is closely related to the Dirichlet perturbation model in Example
\ref{eg:dir.perturbation.maintext}, because the density of $Q$ (with respect to
a suitable reference measure) is proportional to $e^{\frac{1}{\lambda} c(p,
q)}$ \cite[Remark 6]{pw18b}. It is easy to verify that $c(p, q) = {\bf L}_{-1,
\varphi}[q : p]$ where $\varphi(p) = - \sum_{i = 0}^{n-1} \frac{1}{n} \log p_i$
is  $c_{-1}$-convex on $\triangle^n$. Up to a change of variables and addition
of linear terms (see \cite[Remark 3]{w18}), the Dirichlet cost function is
equivalent to the logarithmic cost $c_{-1}$.
The ($-1$)-mirror map then corresponds to the optimal transport map of the
Dirichlet transport.

We now adapt the logarithmic divergence and the $(-1)$-mirror map to the simplex
following the notations of \cite{pw18b}. The role of the $c_{-1}$-convex
generator is now played by an {\it exponentially concave} function.

\begin{definition}[Exponentially concave function]
A smooth function $\varphi: \triangle^n \rightarrow \mathbbm{R}$ is said to be
exponentially concave if $e^{\varphi}$ is concave. Given such a function, we
define its $L$-divergence by
\begin{equation} \label{eqn:L.divergence}
{\bf L}_{\varphi}[q : p] = \log( 1  + \DP{\nabla \varphi(p)}{q - p})
- (\varphi(q) - \varphi(p)).
\end{equation}
\end{definition}

It is easy to see that if $\varphi$ is exponentially concave, then $-\varphi$ is
$c_{-1}$-convex and ${\bf L}_{\varphi} = {\bf L}_{-1, -\varphi}$. In order that
the induced Riemannian metric is well-defined, we assume that $\nabla^2
e^{\varphi}$ is strictly negative definite when restricted to the tangent space
of $\triangle^{n}$. The ($-1$)-mirror map is now given in terms of the optimal
transport map of the Dirichlet transport. We let $e_0, \ldots, e_{n-1}$ denote
the standard Euclidean basis. Given a differentiable function $f$ on
$\triangle^n$, define the directional derivative
\[
\widetilde{\nabla}_i f(p) = \DP{\nabla f(p)}{e_i-p}, \quad i = 0, \ldots, n - 1.
\]
Recall the perturbation operation defined by \eqref{eqn:perturbation}. We also
define the {\it powering operation} for $p \in \triangle^n$ and $\alpha \in
\mathbbm{R}$ by
\[
\alpha \otimes p = \left( \frac{(p^0)^{\alpha}}{\sum_{j = 0}^{n-1}
(p^j)^{\alpha}}, \ldots, \frac{(p^{n-1})^{\alpha}}{\sum_{j = 0}^{n-1}
(p^j)^{\alpha}} \right).
\]
Note that $\triangle^n$ is an $(n - 1)$-dimensional real vector space under the
operations $\oplus$ and $\otimes$. We define $\ominus p = (-1) \otimes p$ be the
additive inverse of $p$.

\begin{definition}[Portfolio and optimal transport maps]
Given the exponentially concave generator $\varphi$, we define the portfolio map
$\pi_{\varphi} : \triangle^n \rightarrow \triangle^n$ by
\begin{equation} \label{eq:portfolio}
(\pi_{\varphi}(p))^i = p^i \left(1 + \widetilde{\nabla}_i \varphi(p) \right),
\quad i = 0, \ldots, n - 1.
\end{equation}
 The optimal transport map $T_{\varphi}: \triangle^n \rightarrow \triangle^n$ is
 defined by
\begin{equation} \label{eqn:Dirichlet.transport.map}
q = T_{\varphi}(p) = p \oplus \pi_{\varphi}(\ominus p).
\end{equation}
\end{definition}

That $T_{\varphi}$ is an optimal transport map for the Dirichlet cost function
\eqref{eqn:Dirichlet.cost} is proved in \cite[Theorem 1]{pw18b}, which is an
analogue of Brenier's theorem. The terminology ``portfolio map'' for the mapping
$\pi_{\varphi}$ is motivated by its use in portfolio theory
\cite{fernholz2002stochastic, pw16, w19}.

\begin{example}[Examples of portfolio and transport maps]
\label{ex:portfolios.examples}
\begin{enumerate}
    \item[(i)] Let $\varphi(p) = \sum_{i = 0}^{n-1} \frac{1}{n} \log p_i$. Then
      the associated portfolio map is the constant map $\pi_{\varphi}(p)
      = \left( \frac{1}{n}, \ldots, \frac{1}{n} \right)$ called the {\it
      equal-weighted portfolio}. From \eqref{eqn:Dirichlet.transport.map}, the
      transport map is the identity $T_{\varphi}(p) = p$. This function
      corresponds to the quadratic function $\frac{1}{2}|x|^2$ whose Euclidean
      gradient is the identity.
    \item[(ii)] Let $\varphi(p) = \frac{1}{\alpha} \log \left( \sum_{j
      = 0}^{n-1} (p^i)^{\alpha} \right)$ where $\alpha \in (-\infty, 1)$ is
      a fixed parameter. Then $\pi_{\varphi}(p) = \alpha \otimes p$ is called
      the {\it diversity-weighted portfolio}. The transport map is given by
      $T_{\varphi}(p) = (1 - \alpha) \otimes p$ which can be interpreted as
      a dilation under the Aitchison geometry. As $\alpha \rightarrow 0$ we
      recover the identity transport.
\end{enumerate}
\end{example}

Let $f: \triangle^n \rightarrow \mathbbm{R}$ be a differentiable function. Using
the Riemannian metric $g$ induced by ${\bf L}_{\varphi}$, we can define the
gradient flow
\begin{equation} \label{eqn:Dirichlet.gradient.flow}
\dv{}{t} p_t = -\mathrm{grad}_{g} f(p_t),
\end{equation}
which is a special case of \eqref{eqn:L.lambda.flow} (up to reparameterization)
when $\lambda = -1$. The following result derives explicitly the dynamics under
the dual variable $q_t = T_{\varphi}(p_t)$ defined in terms of the transport
map. We omit the proof as it is a straightforward, but tedious computation.

\begin{theorem}[Conformal mirror descent on $\triangle^d$ under Dirichlet
  transport]
\label{thm:cmd-dt-supp}
Consider the gradient flow \eqref{eqn:Dirichlet.gradient.flow}, and let $q_t
= T_{\varphi}(p_t)$. Then for $i = 0, \ldots, n - 1$, we have
\begin{equation} \label{eqn:gradient.flow.simplex.done}
	\dv{}{t} \log q_t^i = \frac{-p_t^i}{\pi_{\varphi}^i(\ominus p_t)} \left[
  \widetilde{\nabla}_i f(p_t) - q_t^i \sum_{j = 0}^{n-1} \left(
\frac{p_t^j}{p_t^i}\right)^2 \widetilde{\nabla}_j  f(p_t) \right].
\end{equation}
\end{theorem}

\begin{example}
Consider the equal-weighted portfolio in \eqref{ex:portfolios.examples}. Then
$q_t = T_{\varphi}(p_t) = p_t$ and corresponding gradient flow
\eqref{eqn:gradient.flow.simplex.done} is given by
\[
\dv{t} \log \frac{p_t^i}{p_t^j} = - n \left[ p_t^i \widetilde{\nabla}_i f(p_t)
- p_t^j \widetilde{\nabla}_j f(p_t) \right].
\]
This motivates the multiplicative discrete update
\begin{equation*}
p_{k+1}^i = \frac{ p_k^i \exp \left( -\delta_k p_k^i \widetilde{\nabla}_i
f(p_t)\right)}{\sum_{j = 0}^{n-1}  p_k^j \exp \left( -\delta_k p_k^j
\widetilde{\nabla}_j f(p_t)\right)}.
\end{equation*}
This is similar to but different from the {\it entropic descent} (Bregman mirror
descent on $\triangle^n$ induced by the negative Shannon entropy) whose update
is given by
\begin{equation} \label{eqn:entropic.descent}
p_{k+1}^i = \frac{ p_k^i \exp \left( -\delta_k \nabla_i f(p_t)\right)}{\sum_{j
= 0}^{n-1}  p_k^j \exp \left( -\delta_k \nabla_j f(p_t)\right)},
\end{equation}
where $\nabla_i f$ is the $i$-th component of $\nabla f$.
\end{example}

\begin{figure}[t!]
  \centering
  \begin{subfigure}{0.49\textwidth}
  \begin{center}
    \includegraphics[width=\textwidth]{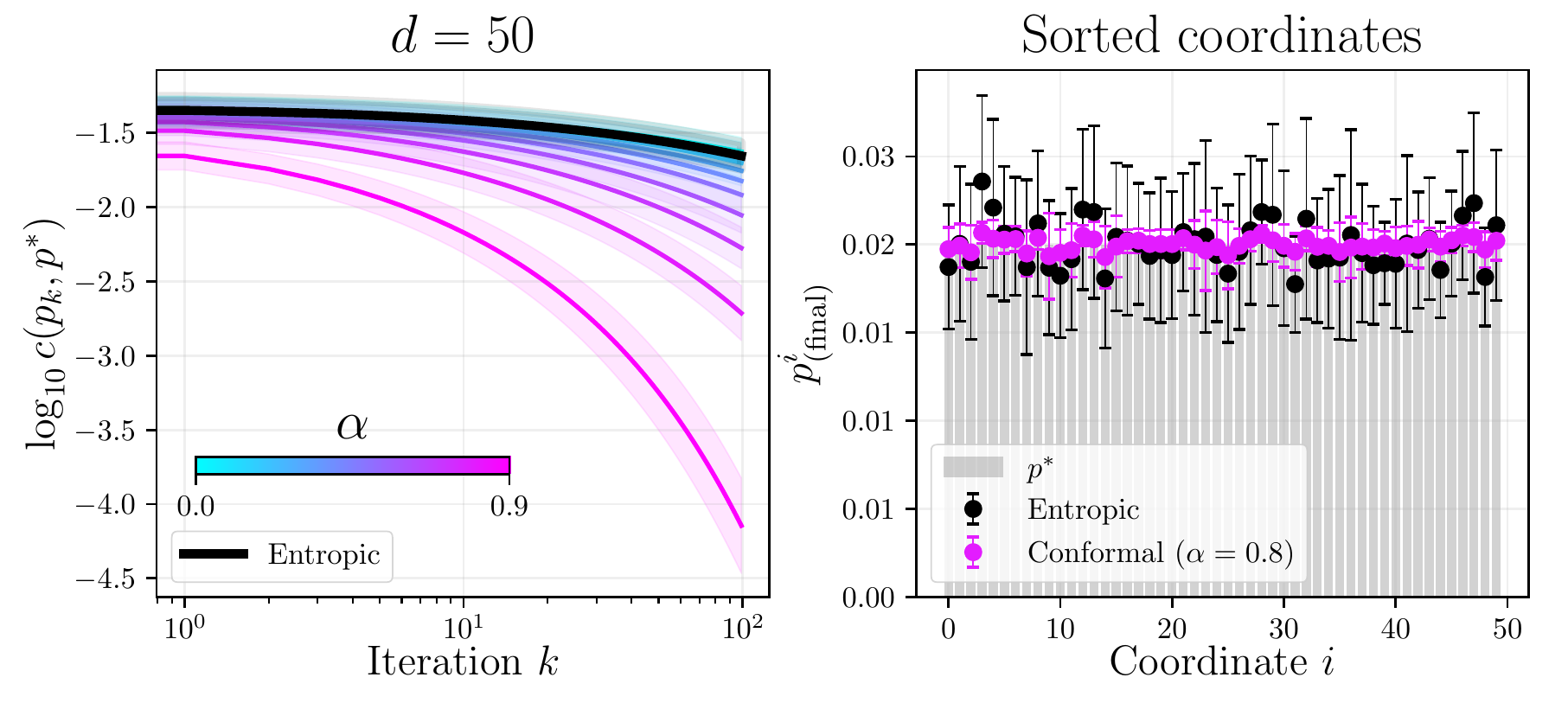}
  \end{center}
  \caption{}
  \label{fig:md-dc-center}
  \end{subfigure}
  \begin{subfigure}{0.49\textwidth}
  \begin{center}
    \includegraphics[width=\textwidth]{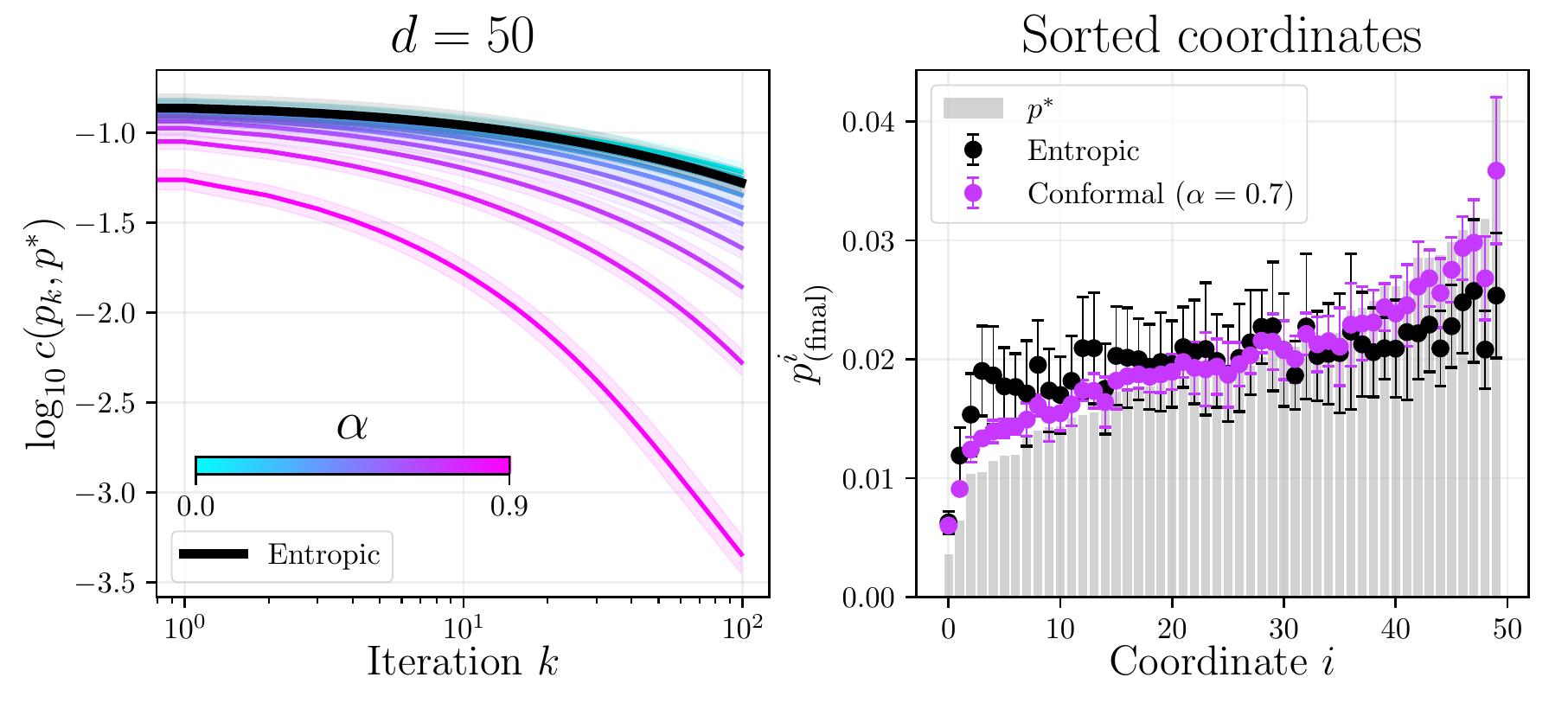}
  \end{center}
  \caption{}
  \label{fig:md-dc-dirichlet-2.pdf}
  \end{subfigure}
  \begin{subfigure}{0.49\textwidth}
  \begin{center}
    \includegraphics[width=\textwidth]{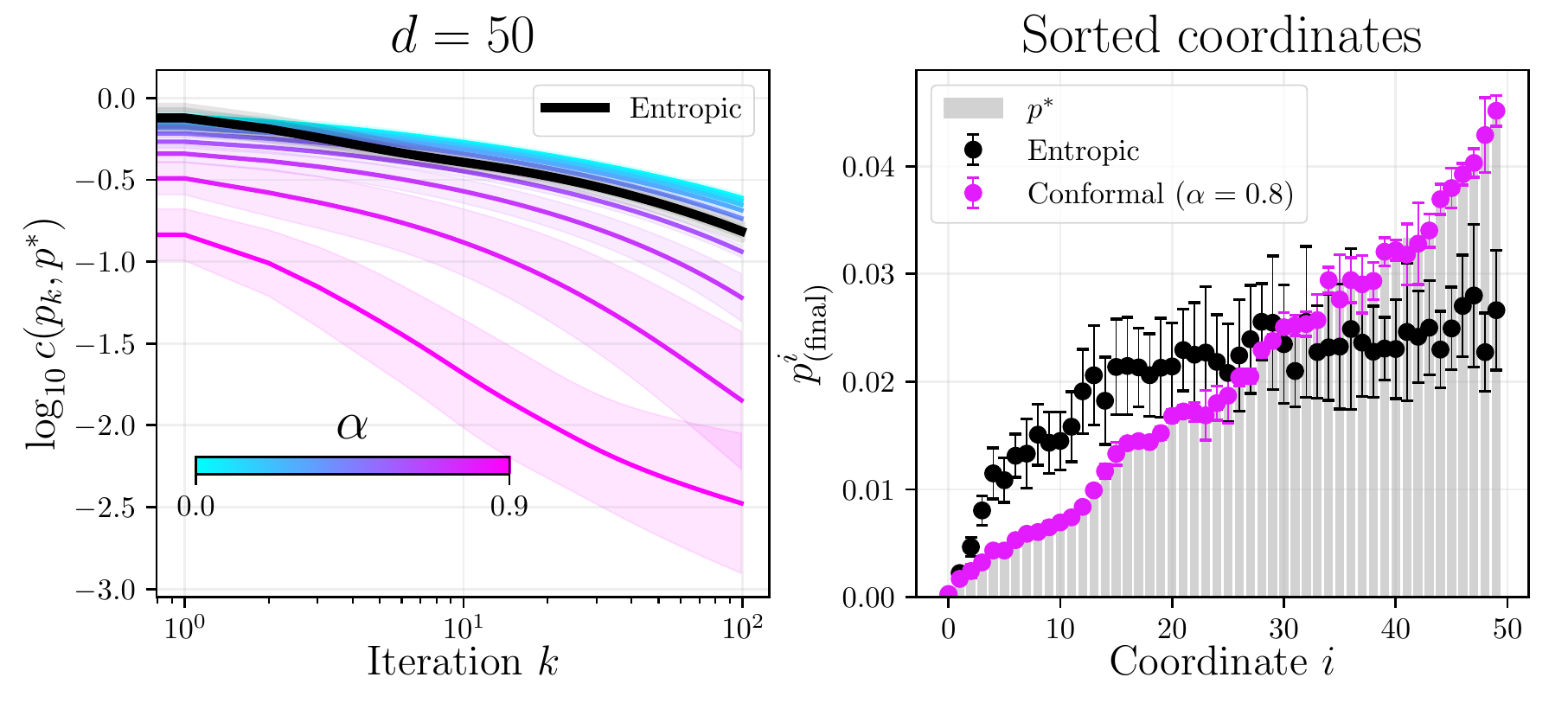}
  \end{center}
  \caption{}
  \label{fig:md-dc-dirichlet-random}
  \end{subfigure}
  \begin{subfigure}{0.49\textwidth}
  \begin{center}
    \includegraphics[width=\textwidth]{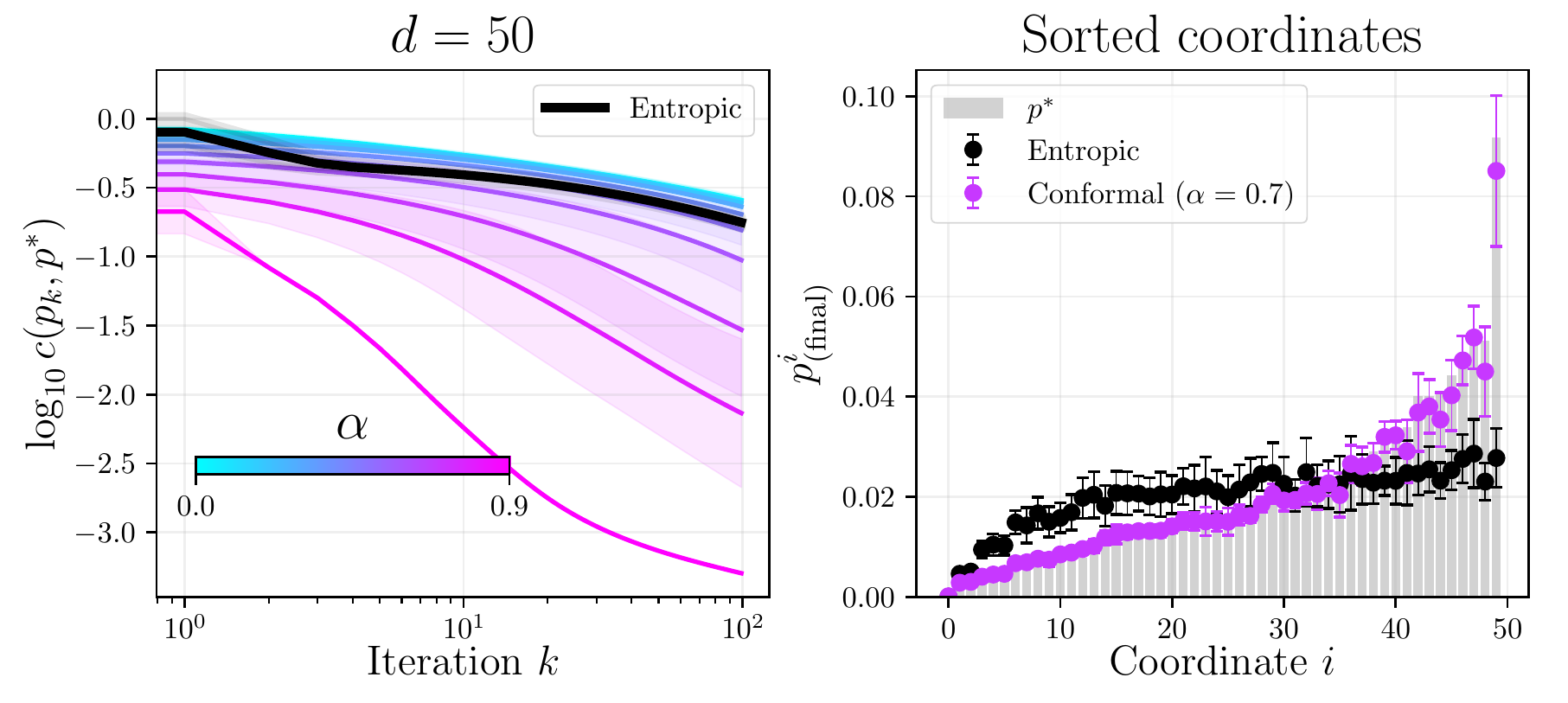}
  \end{center}
  \caption{}
  \label{fig:md-dc-dirichlet-1.pdf}
  \end{subfigure}
  \caption{Convergence rates (left) and final estimates $p_{(\text{final})}$
  (right) of $f(p_k) = c(p_k, p^*)$ for the entropic and conformal descents
  using step size $\delta_k = \frac{1}{d\sqrt{k}}$, for various targets $p^*$ that
  were randomly chosen and fixed. In Figure~\ref{fig:md-dc-center}, $p^*$ is the
  barycenter on $\triangle^d$, whereas in Figures
  ~\ref{fig:md-dc-dirichlet-2.pdf}-\ref{fig:md-dc-dirichlet-1.pdf}, $p^*$ was
  sampled from a Dirchlet distribution with varying parameters.
  For both the entropic and conformal descents, we plot the average over $12$
  randomized initial points $p_0$ (for $p^*$ fixed).}
  \label{fig:md-dc}
\end{figure}

\begin{example}
Consider minimization of the function $f(p) = c(p, p^*)$ where $c$ is the
Dirichlet cost function defined by \eqref{eqn:Dirichlet.cost} and $p^*$ is
fixed. In this experiment, we generate $p^*$ randomly according to various
distributions on $\triangle^n$. We implement
\eqref{eqn:gradient.flow.simplex.done} using the forward Euler discretization
for the diversity-weighted portfolio (Example \ref{ex:portfolios.examples}(ii))
where $\alpha \in \{0, \ldots, 0.9\}$, and compare the performance with that of
the entropic descent \eqref{eqn:entropic.descent}. The results are shown in
Figure~\ref{fig:md-dc}. Values of $\alpha$ closer to $1$ perform better than the
entropic descent across all settings, and recover the minimizer $p^*$
considerably more accurately.
\end{example}

\section{Discussion and future directions} \label{sec:conclusion}
Convex duality and Bregman divergence underlie much of the theory and
applications of classical information geometry. In this paper, we use the
$\lambda$-duality and the associated logarithmic divergence to propose
a tractable gradient flow called the conformal mirror descent. We demonstrate
its usefulness in online parameter estimation and gradient flows on the simplex.
Here, we discuss other related work and some directions for future research.

In this paper, we generalize the Hessian gradient flow primarily from the
information-geometric point of view. Being a fundamental first-order
optimization method, mirror descent has been studied and generalized in many
directions. For instance, convergence of many discrete and continuous time
descent algorithms was studied using Lyapunov arguments in \cite{wilson18}. In
\cite{krichene2015acceleratedmirror}, a family of accelerated mirror descent
algorithms with quadratic convergence was proposed. Likewise,
\cite{wibisono2016variational} presents a unifying analysis of accelerated
descent using variational methods. A future avenue is to explore accelerated
variants of the conformal mirror flow, and to interpret these using
information-geometric frameworks; one such exploration is presented by
\cite{defazio2019curved}.

Mirror descent provides a concrete framework to understand seemingly unrelated
optimization algorithms. Recent lines of work
\cite{mishchenko2019sinkhorn,leger2020agradient,mensch2020onlinesinkhorn} have
analyzed and interpreted the popular Sinkhorn algorithm
\cite{10.1214/aoms/1177703591} -- an iterative scheme used for solving the
entropic optimal transport problem \cite{cuturi2013sinkhorndistances} -- as
a form of mirror descent. Our conformal mirror descent may be applied to develop
new algorithms for regularized optimal transport problems and analyzing their
convergence properties.

Statistical inference and machine learning involving generalized exponential
families is the subject of a recent line of work, for
e.g.~\cite{ding2013statistical, gayen2021projection}. We expect that
$\lambda$-duality and logarithmic divergences will be useful in this endeavor.
Nevertheless, the current framework (as in \cite{WZ21}) assumes that the
curvature parameter $\lambda$ is given and known (except in special cases such
as the Dirichlet perturbation model in Example
\ref{eg:dir.perturbation.maintext}). A natural direction is to develop
data-driven methods to choose $\lambda$ (and analogous quantities for other
generalized exponential families).

The $\lambda$-duality is a special case of the $c$-duality in optimal transport
when $c = c_{\lambda}$ is the logarithmic cost given by \eqref{eqn:log.cost}.
While the $\lambda$-duality is particularly tractable, efficient algorithms
related to general $c$-duality will likely open up many new applications. For
example, the recent paper \cite{CAL21} used $c$-convexity to define normalizing
flows on Riemannian manifolds. It is also natural to analyze similarly derived
gradient flows with respect to other cost functions. We hope our results will
motivate and inspire further work in applications of generalized $c$-convex
duality.

\section*{Acknowledgement}
The research of L.~Wong is partially supported by an NSERC Discovery Grant
(RGPIN-2019-04419) and a Connaught New Researcher Award. F.~Rudzicz is partially
supported by a CIFAR Chair in Artificial Intelligence.

\section*{Data availability}
The data used in this paper was simulated and the codes are available upon
request.

\bibliographystyle{abbrv}
\small{\bibliography{references}}

\end{document}